\pdfoutput=1
\documentclass{amsart}

\usepackage[breakable,skins]{tcolorbox}
\newtcolorbox{tbox}[1][]{%
    breakable,
    enhanced,
    colframe=blue,
    coltitle=white,
    #1
}

\usepackage{amssymb}
\usepackage{booktabs}
\usepackage{hyperref}
\usepackage{todonotes}
\hypersetup{
    colorlinks=true,       
    linkcolor=blue,          
    citecolor=blue,        
    filecolor=blue,      
    urlcolor=blue           
}
\usepackage[alphabetic,backrefs,lite]{amsrefs}
\usepackage{verbatim}

\usepackage[all]{xy}

\numberwithin{equation}{section}

\theoremstyle{plain}
\newtheorem{lemma}{Lemma}[section]
\newtheorem{theorem}[lemma]{Theorem}
\newtheorem{introthm}{Theorem}
\newtheorem{proposition}[lemma]{Proposition}
\newtheorem{corollary}[lemma]{Corollary}

\theoremstyle{definition}
\newtheorem{definition}[lemma]{Definition}

\newtheorem{question}[]{Question}
\newtheorem{property}[]{Property}
\newtheorem{propA}[]{Proposition}

\newtheorem{remark}[lemma]{Remark}
\newtheorem{notation}[lemma]{Notation}

\DeclareMathOperator{\vdim}{vdim}
\DeclareMathOperator{\edim}{edim}

\newcommand{\ls}{\mathcal{L}}

\newcommand{\LL}{{\mathcal L}}
\newcommand{\PP}{{\mathbb P}}

\newcommand{\ldim}{\operatorname{ldim}}

\begin{document}

\title[]{On linear systems with multiple points
on a rational normal curve}
\author{Antonio Laface}
\address{
Departamento de Matem\'atica \newline
Universidad de Concepci\'on \newline 
Casilla 160-C \newline
Concepci\'on, Chile}
\email{alaface@udec.cl}

\author{Elisa Postinghel}
\address{Dipartimento di Matematica\newline Universit\`a degli Studi di Trento \newline via Sommarive 14
I-38123 \newline Povo di Trento (TN), Italy}
\email{elisa.postinghel@unitn.it}

\author{Luis Jos\'e Santana S\'anchez}
\address{Department of Mathematical Sciences \newline 
Loughborough University,  \newline
Epinal Way,
Loughborough LE11 3TU, 
 \newline
 United Kingdom \newline
\underline{Current}:
Departamento de Matem\'aticas, Estad\'istica e I. O. \newline 
Universidad de La Laguna \newline
Apartado de Correos 456, 38200 \newline
San Cristóbal de La Laguna, Spain.}
\email{lsantans.@ull.edu.es}


\thanks{The first author has been partially
supported by Proyecto FONDECYT Regular n. 1190777.
The second author is a members of INdAM-GNSAGA and she
was partially supported by the EPSRC grant EP/S004130/1.
The third author has been partially
supported by Proyecto FONDECYT Regular n. 1190777.}

\maketitle 

\begin{abstract}
We give a closed formula for the dimension of all linear systems in $\mathbb{P}^n$ with assigned multiplicity at arbitrary collections of points lying on a rational normal curve of degree $n$. 
In particular we give a purely geometric explanation of the speciality of these linear systems, which is due to the presence of certain subvarieties in the base locus: linear spans of points, secant varieties of the rational normal curve or joins between them.

 \end{abstract}

\section*{Introduction}

It is a classical question to compute the dimension of a linear system $\mathcal{L}=\mathcal{L}_{n,d}(m_1,\dots,m_s)$ of
hypersurfaces of degree $d$ in $\mathbb{P}^n$ passing through a collection
of $s$ points in general position with multiplicities
at least $m_1,\dots,m_s$. Aside from  some partial results, the question is open in general.
If we define the (affine)  {\em virtual dimension} $\vdim(\LL)$ to be the integer
\[
\vdim(\LL) =\binom{n+d}{n}-\sum_{i=1}^s\binom{n+m_i-1}{n},
\]
obtained by a parameter count, a linear system $\ls$ is said to be {\it non-special} if it has the (affine) expected dimension, 
which is $\edim(\ls) = \max(\vdim(\ls), 0)$, or {\it non-special} if $\dim(\ls)>\edim(\ls)$.

In general, computing the dimension of these linear systems is a challenging task. 
In order to classify the special ones, one has to understand first the base locus and in particular what are the subvarieties of $\PP^n$
 that induce speciality when contained with
multiplicity high enough in the base locus of $\ls$.  We shall refer to these varieties as \emph{special effect varieties}, using the terminology introduced in \cite{bocci}.

For points in general position, the only complete result is the celebrated Alexander-Hirschowitz Theorem (\cite{AlHi}, see also \cite{BO}, \cite{Po}), which provides a complete list of special systems in the case where all multiplicities equal $2$.
For higher multiplicities, recall that $\mathcal{L}_{n,d}(m_1,\dots,m_s)$
is {\em Cremona reduced} if $(n-1)d\geq m_1+\cdots+m_{n+1}$
and $m_1\geq\cdots\geq m_s$ and that any linear system can
be put in this form after applying a Cremona transformation.
The planar case, i.e. $n=2$, has been deeply investigated by many authors starting from the classic Italian school of Algebraic Geometry. A famous conjecture, often referred to as the Segre-Harbourne-Gimigliano-Hirschowitz Conjecture, 
says that a linear system 
is special if and only if its strict transform on the blow-up of $\PP^2$ at the points has a $(-1)$-curve at least doubly contained in its base locus
\cite{S,  Harbourne3, G, Hir2} (see also \cite{Ciliberto}), that is to say that $(-1)$-curves are the only special effect curves.
Equivalently the conjecture says that a Cremona reduced linear system
is non-special. For $n=3$ the Laface-Ugaglia conjecture \cite[Conjecture 4.1]{LU} says 
that special effect varieties for Cremona reduced linear 
systems are: lines through two points and quadrics
through nine points.
For the higher dimensional case, the contribution of linear spaces, spanned by sets of points in general position, was studied by the second author and others in \cite{BDPLinear}, \cite{DPvanishing}. In particular it is proved that in dimension $n$ the special effect subvarieties are all and only these linear cycles  when either the number of points is bounded above by $n+2$, or the number of points is arbitrarily large but the multiplicities are bounded  with respect to the degree, 
i.e. if $\sum_im_i\leq nd+\min\{n-s(d),s-n-2\}$, with $s(d)
:= \#\{m_i\, :\, d=m_i\}$.

In \cite{Conjecture} the case of $n+3$ points in linearly general position of $\PP^n$ was investigated in detail. Recalling the classical result, attributed to Castelnuovo or Veronese, that through $n+3$ points there is a unique rational normal curve of degree $n$,
the authors conjectured that in this case special effect varieties
are: linear spans of points, secant varieties of the rational normal curve, or joins between them. Moreover a conjectural formula for the dimension of any linear system $\mathcal{L}$
was proposed. 

For points in special configuration the dimensionality problem can become very wild and no general results are known. In this paper we consider the case of $s$ points $p_1, \ldots, p_s$ lying on a rational normal curve of degree $n$ of $\PP^n$, with $s\ge n+3$ arbitrarily large. 
We give a closed formula for the dimension of any linear system $\ls$ 
and we show that the special effect subvarieties are 
all and only: the linear spans of points $L_I = \langle p_i \rangle_{i\in I}\subsetneq\PP^n$, where $I\subset \{1, \ldots, s\}$, the rational normal curve $C$, its secant varieties $\sigma_t=\sigma_t(C)$ and all the joins between the above $J(L_I, \sigma_t)$. We denote by $r_{I,\sigma_t}$ the dimension of $J(L_I, \sigma_t)$ and by $k_C$ and $k_{I,\sigma_t}$ the expected multiplicities of containment  of $C$ and of $J(L_I, \sigma_t)$ respectively in the base locus of $\mathcal L$ (cf. Section \ref{special-effect}).
Our main theorem is the following, for a precise statement see Theorem \ref{ThmC}.

\begin{introthm}\label{ThmCintro}
Let $\mathcal L := \mathcal L_{n,d}(m_1,\dots,m_s)$ be a linear system with $s\geq n+3$ points lying on a rational
normal curve of degree $n$, with $m_1, \ldots, m_s \geq k_C$. Then 
\[
 \dim(\mathcal L) = \vdim(\mathcal L) + \sum (-1)^{|I|}F_t (n + k_{I, \sigma_t} - r_{I, \sigma_t} - 1, s, \varepsilon, n), 
\]
where the summation ranges over the list of special effects varieties, the functions $F_t$ are recursively defined as
$$ F_t(a, s, \varepsilon , n) := \binom{a}{n} + \sum_{i=1}^t \binom{s-n-4 + i}{i} \binom{a+i}{n} - \sum_{i=1}^t \binom{\varepsilon}{i}F_{t-i}(a,s,\varepsilon, n-i) $$
and $\varepsilon$ is a parameter depending on $n, d, m_1, \ldots, m_s$, defined in \eqref{kC with epsilon}. 
\end{introthm}

Thanks to the simple observation that if $m_i < k_C$ for some $i$, then $\dim (\mathcal L) = \dim (\mathcal L_{n,d}(m_1, \ldots, \check m_i, \ldots , m_s))$, Theorem \ref{ThmCintro} completely solves the dimensionality probem. This in particular establishes the conjecture of \cite[Conjecture 6.4]{Conjecture} and extends its statament to the case of an arbitrary number of points.

This special configuration of points has been widely studied also in the Commutative Algebra framework.  A conjectural algorithm that computes the Hilbert function of fat point schemes with support on a rational normal curve of degree $n$ in $\PP^n$ was proposed by Catalisano, Elli\'a and Gimigliano in \cite[Conjecture C]{CEG}. Since one can show that computing such Hilbert 
function is equivalent to solving the dimensionality problem for  $\ls=\mathcal L_{n,d}(m_1, \ldots, m_s)$, their work implies a conjectural algorithm for computing $\dim(\ls)$. Moreover, Catalisano, Trung and Valla in \cite[Proposition 7]{CTV} gave an explicit formula for the \textit{regularity index} of the ideal sheaf of a fat point scheme supported on a rational normal curve.  In terms of linear systems, this translates into finding  the minimum  $\delta \in \mathbb N$ for which $\mathcal L_{n,d}(m_1, \ldots, m_s)$ is non-special for every $d\geq \delta$. 

More recently, in \cite[Problem 6.4]{HM}, H\`a and Mantero raised the question as to whether one could obtain an Alexander-Hirschowitz type theorem for points on a rational normal curve of degree $n$ of $\PP^n$, that is to say, to provide a formula for $\dim(\ls)$, when  $m_1=\cdots =m_s=2$. Our Theorem \ref{ThmCintro} answers this question.


Linear systems $\mathcal L_{n,d}(m_1, \ldots, m_s)$  with points on a rational normal curve $C$ of degree $n$ are interesting also from the birational geometric perspective.
The blow-up of the projective space at $s$ points lying on $C$, that we denote by $X^n_s$, is a Mori dream space, i.e. its Cox ring is finitely generated as shown in \cite[Theorem 1.2]{CT}. A list of generators includes the $s$ exceptional divisors and the strict transforms of all joins $J(L_I,\sigma_t)\subset\PP^n$ of codimension $1$. In order to prove their result, Castravet and Tevelev used certain restriction sequences of the space of global sections of a Cartier divisor on $X^n_s$ to an exceptional divisor, that can be identified with the space $X^{n-1}_s$, and they showed that Cox ring generators of $X^{n-1}_s$ lift to Cox ring generators of $X^{n}_s$, cf. Section \ref{CT construction}. We observe that the latter in particular establishes the conjectural algebraic algorithm of \cite{CEG} via a completely skew approach.
This will be the key tool in the inductive procedure that we will adopt to prove our main theorem, Theorem \ref{ThmCintro}. Hence, besides its original motivation in the setting of linear systems, our work has the twofold purpose to conclude an investigation started by the algebraists concerning fat point schemes supported on $C$ and to give it a new geometric perspective in the language of Mori dream spaces. It also gives a new geometric interpretation to results of \cite{CTV} which, in our language, reads as: the regularity index of a collection of points of multiplicity $m_1,\dots,m_s$ on $C$ is the smallest integer $\delta$ such that $\mathcal L_{n,d}(m_1, \ldots, m_s)$ has no special effect subvarieties, for every $d\ge \delta$.

Special effect varieties for points on a rational normal curve shed light on further birational properties of the Mori dream space $X^n_s$.
We recall that if $s\le n+1$, the latter is a toric variety, while if $s=n+3$ then  $X^n_{n+3}$ has a nice interpretation as the moduli space of certain rank-$2$ parabolic vector bundles over a $(n+3)$-pointed line $\PP^1$, see \cite{Bau91}, \cite{Mukai05} for details. Moreover, using this interpretation, Mukai  showed that the Mori chamber decomposition of the effective and movable cones of divisors of $X^n_{n+3}$ is induced by a hyperplane arrangement in the N\'eron-Severi space and he described the equations in the $n+4$ variables corresponding to a basis of the Picard group, see also \cite[Section 3]{AM}. Since such equations correspond to setting to zero the linear formulas for the multiplicities of containment of the special effect varieties in the stable base locus of effective divisors computed in \cite[Lemma 4.1]{Conjecture}, we can say that the  birational geometry of $X^n_{n+3}$ is governed by the special effect subvarieties. More precisely, the fixed divisors among the generators of the effective cone, besides the exceptional divisors, are all and only the strict transforms of the special effect hypersurfaces of $\PP^n$. Furthermore, the small $\mathbb{Q}$-factorial modifications of $X^n_{n+3}$ are compositions of  flips of strict transforms of special effect subvarieties of higher codimension and the base loci of any effective divisor is stable. In particular the Mori chamber decomposition of the effective cone of $X^n_{n+3}$ is induced by the stable base locus in the sense that all divisors that are ample on a small modification of $X^n_{n+3}$  have the same base locus support in $X$. Another way to say it is: the Mori chamber decomposition and the stable base locus decomposition of the effective cone coincide, see \cite[Chapter 2]{BCP} for details. Since these properties hold for $s=n+3$, they also hold for $s\le n+2$. In particular in the toric case, $s\le n+1$, the special effect varieties are all and only the torus invariant subvarieties that are all linear cycles. 
 For $s>n+3$ the effective cone is generated by the special effect divisors and by the exceptional divisors, as shown in \cite[Theorem 1.2]{CT}. The following question arise naturally.
\begin{question}
Let $X$ be the blow-up of $\PP^n$ at $s$ points lying on a rational normal curve of degree $n$.  What is the nef cone of $X$? Do the set of special effect varieties coincide with the set of flippable cycles? In particular, do the Mori chamber decomposition and the stable base locus decomposition of $X$ coincide? \end{question}

The article is organised as follows.
In Section \ref{preliminaries} we give a description of the subvarieties of $\PP^n$ that are special effect for linear systems $\ls_{n,d}(m_1,\dots,m_s)$ for $s$ points lying on a rational normal curve of degree $n$, and we recall the description of the Cox ring of the blow-up $X^n_s$ in terms of its generators.

In Section \ref{MainTheorem} we give the statement of the main theorem, Theorem \ref{ThmC}, which gives a formula for the dimension of the space of global sections of all effective divisors on $X^n_s$ or, equivalently, the solution to the dimensionality problem for linear systems  $\ls_{n,d}(m_1,\dots,m_s)$. 
Moreover the formula for the dimensionality is implemented in Magma \cite{Magma}, see Section \ref{magma-section}.

Section \ref{ProofMainTheorem} contains a number of intermediate results (Propositions  \ref{ThmdimP2}-\ref{propB}-\ref{propC}-\ref{propD}) on which the proof of the main result is based via an intricate induction argument.

Finally, appendix \ref{Fiveproperties} contains the proof of a number of technical properties used in the main arguments.


\section{Preliminaries}\label{preliminaries} Throughout this paper, we will work over the field of complex numbers. 
Let $C\subset \PP^n$ be a rational normal curve of degree $n$ and let  $p_1,\dots, p_s\in C$ be $s$ different points, with $s\in\mathbb{N}$ an arbitrary number.
In this section we introduce our notation and we will set up the main techniques.

\subsection{Blow-ups and divisors}
Let $\pi : X^n_s \rightarrow \mathbb P^n$ denote the blow-up of $\mathbb P^n$ at $p_1, \ldots, p_s$ with exceptional divisors $E_1,\dots, E_s$. The latter, together with the class $H$ pull-back of a generic hyperplane of $\PP^n$, generate the Picard group of $X^n_s$.
 Then, the linear system  $\mathcal L = \mathcal L_{n,d}(m_1, \ldots, m_s)$ pulls back via $\pi$ to the complete linear system of the following divisor on $X^n_s$:
\begin{equation}\label{D general form}  D:= dH - \sum_{i=1}^s m_i E_i. \end{equation}
In particular the following identities hold
\begin{align*}
\dim(\LL)&=h^0(X^n_s, \mathcal O_{X_s^n}(D)),\\
\vdim(\LL)&=h^0(X^n_s, \mathcal O_{X_s^n}(D))-h^1(X^n_s, \mathcal O_{X_s^n}(D)),
\end{align*}
where  we used the abbreviations  $h^i(X^n_s, \mathcal O_{X_s^n}(D))=\dim H^i(X^n_s, \mathcal O_{X_s^n}(D))$.

\subsection{Special effect subvarieties}\label{special-effect}
For any subset $I\subseteq \{1, \ldots, s\}$, let $L_I$ describe the linear subspace of $\PP^n$ spanned by the points indexed by $I$. Since the points are in linearly general position, then $\dim (L_I)=|I|-1$ if $|I|\le n$ and $L_I=\PP^n$ otherwise.

For every  $t\geq 1$ we denote with $\sigma_t:=\sigma_t(C)\subseteq \mathbb{P}^n$ the $t$-\emph{secant variety} $C$, that is the Zariski closure of the union of all linear spaces spanned by $t$ points of C, where, in particular, $\sigma_1 = C$. We have that $\dim\sigma_t(C)=2t-1$ if $2t-1 < n $ and $\sigma_t(C)=\PP^n$ otherwise.

Finally, we denote by 
$J(L_I, \sigma_t)$
the join of the linear cycle $L_I$ and the secant variety $\sigma_t$, that is the union of all the lines spanned by a point on $L_I$ and a point on $\sigma_t$. For instance $J(p_i,C)=J(L_{\{i\}},C)$ is a pointed cone over $C$. We adopt the  conventions: 
$ J(L_I, \sigma_0) = L_I  $ and  $J(L_\emptyset,\sigma_t) =\sigma_t.$ 

The dimension of all subvarieties $J(L_I, \sigma_t)$ is computed by the following formula:
\begin{equation}\label{dim-cones}r_{I,\sigma_t} := \dim(J(L_I, \sigma_t)) = |I| + 2t-1. \end{equation}
Abusing notation, we will denote by $J(L_I, \sigma_t)$ the subvarieties of $\PP^n$ as well as their strict transforms in $X_s^n$.

In what follows, we will compute a lower bound for the multiplicity of containment of each subvariety $J(L_I, \sigma_t)$ in the base locus of an effective divisor $D$.

\begin{notation}
For any integer $k\in \mathbb Z$, we let $k^+ := \text{max}(k, 0)$.
\end{notation}

\begin{lemma}(\cite[Lemma 4.2]{CT}) \label{boundkc}  Let $D=dH-\sum_{i=1}^s m_i E_i$ be and effective divisor on $X_s^n$. Then, the rational normal curve $C$ is contained in the base locus of $|D|$ with multiplicity at least $k_C^+$, where 
\begin{equation}\label{multC} k_C=k_C(D) :=  \left\lceil \frac{\sum_{i=1}^s m_i - nd}{s-n-2} \right\rceil. \end{equation}
\end{lemma}

\begin{corollary}\label{cor-multjoins}
In the same notation as Lemma  \ref{boundkc}, the subvariety $J(L_I, \sigma_t)$ is contained in the base locus of $|D|$ with multiplicity at least $k_{I,\sigma_t}^+$, where
\begin{equation} \label{multjoins} k_{I,\sigma_t}=k_{I,\sigma_t}(D) := \sum_{i\in I} m_i + tk_C - (t+|I|-1)d. \end{equation}
\end{corollary}
\begin{proof}
The proof follows the same observation used in \cite[Lemma 4.1]{Conjecture}. We recall it here for the sake of completeness. Any  $t$-secant $(t-1)$-plane to $C$ will be contained in the base locus of $D$  at least $tk_C-(t-1)d$ times; moreover, since the multiplicity is semi-continuous, all limits of $t$-secant $(t-1)$-planes will be contained in the base locus of $D$  at least with the same multiplicity. This proves the statement for the secant variety $\sigma_t$. Moreover since every line spanned by a point of $L_I$ and a point of $\sigma_t$ is contained at least $k_I+k_{\sigma_t}-d$ times, the statement holds for all joins $J(L_I,\sigma_t)$.
\end{proof}

We introduce a further parameter, that will be needed in the next sections.
\begin{definition}\label{def-epsilon}
Let  $D=dH-\sum_{i=1}^s m_i E_i$ be an effective divisor in $X_s^n$. We denote by $\varepsilon = \varepsilon(D)$ the \textit{exceeding number} of $D$, which we define as the unique number $\varepsilon \in \{0, \ldots, s-n-3\}$ such that
\begin{equation}\label{kC with epsilon} k_C= \left\lceil \frac{\sum_{i=1}^s m_i - nd}{s-n-2} \right\rceil = \frac{\sum_{i=1}^s m_i - nd + \varepsilon}{s-n-2}.\end{equation}
\end{definition}

\begin{remark} \label{redundant} A crucial first remark that follows from Lemma \ref{boundkc} is that, when $D$ is an effective divisor such that $k_C(D) > m_i >0$  for some index $i \in \{1, \ldots, s\}$ then, the point $p_i$ is \textit{redundant} in the sense that 
$$ H^0(X_s^n, D) \cong H^0(X_{s-1}^n, D + m_i E_i), $$ or, equivalently,
$$\dim (\mathcal L_{n,d}(m_1,  \ldots , m_s)) = \dim (\mathcal L_{n,d}(m_1, \ldots, \check m_i, \ldots , m_s)).$$
By repeating this process, we obtain that any linear system $|D|$ (or $\ls$) is isomorphic to a non-redundant one, that is one for which $m_i\ge k_C$ for every $i\in\{1,\dots,s\}$. \end{remark}

\subsection{Castelnuovo restriction sequences to exceptional divisors}\label{CT construction}  %
In this section we recall the geometric construction used in \cite[Section 4]{CT}. For an effective divisor $D \in \text{Pic}(X_s^n)$, the short exact sequence
$$ 0 \rightarrow \mathcal O_{X_s^n}(D - E_1) \rightarrow \mathcal O_{X_s^n}(D) \rightarrow \mathcal O_{E_1}(D_{\mid E_1}) \rightarrow 0, $$
 yields the exact sequence in cohomology
\begin{align} \label{exactsequence} 
0 \rightarrow H^0(X_s^n, D - E_1) \rightarrow H^0(X_s^n, D) \overset{\rho'}{\rightarrow}  H^0(E_1, D_{|E_1}). 
\end{align}

\noindent The map $\rho'$ is the so called Castelnuovo restriction map and it is in general not surjective. To describe its image, consider the projection $\pi_1: \mathbb P^n \dashrightarrow \mathbb P^{n-1}$  from $p_1 \in \mathbb P^n$ to a general hyperplane $\Pi \cong \mathbb P^{n-1}$. For $i=2, \ldots, s$, the point $p_i$ is mapped to $q_i = \pi_1(p_i)$ that lies on $\pi_1(C)$, which is a rational normal curve of degree $n-1$ in $\Pi$. We will now use the notation $X_{s-1}^{n-1}$ to for the blow-up of $\Pi\cong\mathbb P^{n-1}$  at the points $q_2, \ldots , q_s$. We will write $E_{2,n-1}, \ldots, E_{s,n -1}$ for the exceptional divisors on $X_{s-1}^{n-1}$ and $H_{n-1}$ for the hyperplane class. Then,  defining the linear map $l: \text{Pic}(X_s^n) \longrightarrow \text{Pic}(X_{s-1}^{n-1})$ that sends a divisor   
$$ D= dH-\sum_{i=1}^s m_i E_i \in \text{Pic}(X_s^n)$$
\noindent to the divisor 
\begin{equation} \label{lmap}
 l(D)=m_1 H_{n-1} - \sum_{i=2}^s (m_1 + m_i - d) E_{i,n-1} \in  \text{Pic}(X_{s-1}^{n-1}),
\end{equation}
the map $\rho '$ can be interpreted as a map of spaces of global sections
$$\rho: H^0(X_s^n, D) \rightarrow H^0(X_{s-1}^{n-1}, l(D)).$$

Like $\rho'$, the map $\rho$ is not surjective in general, and this is due to Lemma \ref{boundkc}. Indeed, if we let $q \in \mathbb P^{n-1}$ be the intersection of the tangent line to $C$ at $p_1$ with the hyperplane $\Pi$, then $q \in \pi_1(C)=C_{n-1}$. By Lemma \ref{boundkc}, one concludes that any section in $H^0(X_s^n, D)$ restricts to a section in $H^0(X_s^{n-1}, l(D))$ that vanishes at $q$ with multiplicity at least $k_C^+$. Thus, let us denote by $X_s^{n-1} = \textrm{Bl}_q X_{s-1}^{n-1}$ the blow-up of $X_{s-1}^{n-1}$ at the point $q$, and by $E_{1,n-1}=E_q$ the exceptional divisor. The image of the map $\rho$ is the push-forward of $H^0(X_s^{n-1}, l(D)-k_C^+ E_q)$ to $H^0(X_{s-1}^n, l(D))$, see \cite[Proposition 4.18]{CT}.
From the exact sequence (\ref{exactsequence}), we now obtain the formula:
\begin{equation} \label{indstep} h^0(X_s^n, D) - h^0(X_s^n, D-E_1) = h^0(X_s^{n-1}, l(D) - k_C^+ E_q). \end{equation}

\begin{remark}
Since the authors are not aware of any reference making a connection between the work of Castravet and Tevelev and the conjecture posed by Catalisano, Elli\'a and Gimigliano, we point out that (\ref{indstep}) is equivalent to  \cite[Conjecture C]{CEG}. Hence, Castravet and Tevelev's work \cite{CT} proves this conjecture.
\end{remark}


\section{Statement of the main theorem} \label{MainTheorem}

The goal of this section is to give a closed formula for the dimension of any effective divisor $D$ on $X^n_s$ and to show that it is completely explained by means of the subvarieties $J(L_I, \sigma_t)$ and the corresponding integers $k_{I,\sigma_t}(D)$ defined in Lemma \ref{cor-multjoins}.
The contribution of each such subvariety will be measured by a recursive function that depends on $k_{I,\sigma_t}(D)$, which we are now going to define.

Define $F_t(a, s, \varepsilon, n)$ a function of  five parameters $t, s, \varepsilon, n \in \mathbb N$ and $a \in \mathbb Z$ recursively as follows:
\begin{equation} \label{F_0} F_0(a,s,\varepsilon , n ):= \binom{a}{n}, \end{equation}
\begin{equation} \label{F_t} F_t(a, s, \varepsilon , n) := \binom{a}{n} + \sum_{i=1}^t \binom{s-n-4 + i}{i} \binom{a+i}{n} - \sum_{i=1}^t \binom{\varepsilon}{i}F_{t-i}(a,s,\varepsilon, n-i). \end{equation}
Here we are using the following standard convention: for any  $n\in\mathbb{N}$ and $a,b \in\mathbb{Z}$ with $a<n$ and $b\ge 0$:
$$\binom{a}{n}=0, \quad  \binom{b}{0}=1.$$

\begin{theorem} \label{ThmC} Let $D=dH-\sum_{i=1}^s m_i E_i$ be a non-redundant effective divisor on $X_s^n$ with $s\geq n+3$. Let $\varepsilon$ be the exceeding number of $D$. Then 
\begin{equation} \label{dim} h^0(X_s^n, D)= \sum_{I, \sigma_t} (-1)^{|I|}F_t (n + k_{I, \sigma_t} - r_{I, \sigma_t} - 1, s, \varepsilon, n), \end{equation}
where the sum ranges over all the indices $0\leq t \leq  n/2 $ and $I \subseteq \{1, \ldots , s \}$ such that $0 \leq |I| \leq n-2t$.
\end{theorem}

\begin{remark} Theorem \ref{ThmC} solves the dimensionality problem for all linear systems on $X^n_s$. In fact by Remark \ref{redundant}, if $D$ is redundant, its dimension equals that of a non-redundant one.
\end{remark}

\begin{remark}
We point out that  the contribution to formula \eqref{dim} of the subvariety $J(L_I,\sigma_t)$ is given by the integer
\begin{equation} \label{contribution of join} (-1)^{|I|} F_t (n + k_{I, \sigma_t} - r_{I, \sigma_t} - 1, s, \varepsilon, n), \end{equation}
where $k_{I,\sigma_t}$ is defined in (\ref{multjoins}), 
 $r_{I,\sigma_t}$ is defined in \eqref{dim-cones},  and $\varepsilon$ is as in Definition \ref{def-epsilon}.
 \end{remark}

\begin{remark} The expression on the right hand side of formula \eqref{dim} refines the virtual dimension of $\mathcal L_{n,d}(m_1,\dots,m_s)$, given explicitly in the Introduction. Indeed, if we set  $t=|I|=0$, so that 
$k_{\emptyset, \sigma_0} = d$ and  $r_{\emptyset, \sigma_0}= -1$,
we obtain
$$ (-1)^0F_0(n+d, s , \varepsilon,n) = \binom{n+d}{n}. $$
Similarly, setting $t=0$ and $I=\{i\} \subset \{1, \ldots, s\}$, so that  $k_{\{i\}, \sigma_0}=m_i$ and $r_{\{i\},\sigma_0}=0$, we get 
$$ (-1)^1 F_0 (n + m_i  - 1, s, \varepsilon, n) = - \binom{n+m_i -1}{n}. $$ 
Moreover, if $\mathcal L$ is a non-empty linear system for which $k_{I,\sigma_t} \leq 0$ for every join $J(L_I, \sigma_t)$, then $\mathcal L$ is non-special. In fact one can easily verify that the only non-zero terms of \eqref{dim} are the ones above, so that the dimension of $\ls$ equals its virtual dimension in this case.
\end{remark}

\subsection{Consequences of the main theorem}
As a corollary of Theorem \ref{ThmC} we obtain a different proof of the result of Catalisano, Trung and Valla \cite[Proposition 7]{CTV}. We include it here for the sake of completeness.
\begin{corollary} 
Let $p_1, \ldots, p_s$ be $s$ points on a rational normal curve of degree $n$ in $\mathbb P^n$ and $m_1 \geq \ldots \geq m_s$ positive integers. Then, the linear system $\mathcal L= \mathcal L_{n,d}(m_1, \ldots, m_s)$ is non-special if and only if 
$$d \geq \max \left\{m_1+m_2-1, \left\lfloor \frac{\sum_{i=1}^s m_i + n-2}{n} \right\rfloor \right\}. $$
\end{corollary}
\begin{proof}
For fixed multiplicities $m_1\geq \cdots \geq m_s$ let $D=dH -\sum_{i=1}^s m_i E_i$ be the corresponding divisor on $X_s^n$. First, notice that if $D$ is redundant then the linear system it describes is always special, due to the redundant points. Thus, $d$ needs to be large enough so that $D$ is non-redundant. Under this assumptions Theorem \ref{ThmC} says that $|D|$ is non-special when no special effect variety contributes to its dimension formula. This happens as soon as neither the lines through pairs of points nor the rational normal curve give a contribution.

Now, the contribution of a line $L_{ij}$ through $p_i$ and $p_j$ to the dimension of $|D|$ is computed by the formula $$ F_0(n+k_{ij}-2,s,\varepsilon,n) = \binom{m_i+m_j-d+n-2}{n}. $$
Since we are assuming that $m_1 \geq \ldots \geq m_s$, we conclude that no line will make a contribution if and only if 
\begin{equation} \label{nolines} d \geq m_1 + m_2 -1. \end{equation}

As for the rational normal curve $C$, its contribution to the dimension of $|D|$ is
\begin{equation} \label{contr dim C} F_1(n+k_C-2,s, \varepsilon, n)= \binom{n+k_C-2}{n} +(s-n-3)\binom{n+k_C-1}{n}-\varepsilon \binom{n+k_C-2}{n-1}. \end{equation}
It is clear that, if $k_C\leq 0$, then \eqref{contr dim C} vanishes. Otherwise, we observe that  \eqref{contr dim C}  vanishes if and only if $k_C=1$ and $\varepsilon =s-n-3$, which happens when
$$ \sum_{i=1}^s m_i -nd = 1, $$
by equation (\ref{kC with epsilon}).
Putting these two cases together, we can say that \eqref{contr dim C} vanishes if and only if
\begin{equation} \label{noRNC} d \geq \frac{\sum_{i=1}^s m_i -1}{n}. \end{equation}

The statement follows by putting (\ref{nolines}) and (\ref{noRNC}) together, and noticing that
$$\left\lceil \frac{\sum_{i=1}^s m_i -1}{n}\right\rceil=\left\lfloor \frac{\sum_{i=1}^s m_i + n-2}{n} \right\rfloor.$$ The details of this computation are left to the reader. \end{proof}

The main theorem also implies  an answer to a question posed by H\`a and Montero \cite[Problem 6.4]{HM} for the case of double points $m_1 =\cdots =m_s=2$. 

\begin{corollary} Let $D=dH-\sum_{i=1}^s 2 E_i $ be an effective divisor on $X_s^n$. 
\begin{itemize}
\item If $$d\geq \frac{2s-1}{n},$$ then the linear system $|D|$ is non-special.
\item If 
$$ \frac{s+n+2}{n} \leq d < \frac{2s-1}{n}, $$
then $|D|$ is special and its speciality is given by
\begin{align*}
 h^1(X_s^n,D) &= F_1(n+1-1-1, s, nd-n-s-2, n) \\
			&= 2s -nd -1.
\end{align*}
\item If 
$$ d < \frac{s+n+2}{n} $$
then $|D|$ is special and its speciality is given by
\begin{align*}
 h^1(X_s^n,D) &= F_1(n+2-1-1, s, n(d-2)-4, n) \\
			&= s(n+1)-n^2(d-1)-2.
\end{align*}
\end{itemize}
\end{corollary}

\subsection{Magma library for Theorem \ref{ThmC} and an explicit example}\label{magma-section}
We have implemented a Magma library \cite{Magma} that computes the dimension of any given linear system. This library can be found at the following link:
\begin{center}
\url{https://github.com/alaface/PtsRatNormal}
\end{center}
In this section we work out the statement of Theorem \ref{ThmC} as well as the details of this library in a specific example. Let $$\mathcal L= \mathcal L_{5,8}(7,6^2, 5^7, 2^3)$$ denote a linear system  of hypersurfaces of degree $8$ of $\mathbb P^5$ with $13$ multiple points on a rational normal quintic curve.
We first plug in the data in the Magma code to compute the parameter $k_C$.

\begin{tbox}
\begin{verbatim}
> load "library.m";
> n := 5;
> d := 8;
> m := [7,6,6,5,5,5,5,5,5,5,2,2,2];
> kc(n,d,m); 
4
\end{verbatim}
\end{tbox}
We observe that $\mathcal L$ is a \emph{redundant} linear system since the three last points have multiplicity lower than $k_C$. Hence we may consider  $$\mathcal L' = \mathcal L_{5,8}(7,6^2,5^7),$$ and we have that $\dim(\mathcal L)= \dim(\mathcal L')$. We plug in the parameters of $\ls'$  and compute $k_C$.

\begin{tbox}
\begin{verbatim}
> n := 5;
> d := 8;
> m := [7,6,6,5,5,5,5,5,5,5];
> kc(n,d,m); 
5
\end{verbatim}
\end{tbox}
We can see that $\mathcal L'$  is non-redundant  since every point has multiplicity at least $k_C$. Thus, we may proceed to apply Theorem \ref{ThmC} to compute its dimension.

To do so, we first need to check which joins $J(L_I,\sigma_t)$, for $0\leq t\leq n/2$ and $I\subseteq \{1, \ldots, s\}$ with $0\leq |I|\leq n-2t$, are special effect varieties for $\mathcal L'$.  We look at those for which $k_{I,\sigma_t}> 0$, where the integer is computed in \eqref{multjoins}. For instance, we observe that the join $J(L_{\{2\}}, \sigma_1)$ is contained in the base locus of $\mathcal L'$ since,
$$k_{\{2\},\sigma_1} = 6+5-(1+1-1)8=3. $$
We may check this using the following function of the magma library:
\begin{tbox}
\begin{verbatim}
> ka(n,d,m,{2},1); // introducing the parameters (n,d,m,I,t)
3
\end{verbatim}
\end{tbox}

The complete list of joins $J(L_I, \sigma_t)$ with positive $k_{I,\sigma_t}$ for the linear system $\mathcal L'$, is:
\begin{itemize}
\item[\underline{curves}:] Every line $L_{\{i,j\}}$, where $1\leq i<j \leq 10$. \\
					The rational normal curve $C$.
\item[\underline{surfaces}:] Every cone $J(p_i, C)$, where $1\leq i \leq 10$. \\
					The $2$-planes $L_{\{i,j,k\}}$ where $1\leq i < j \leq 3$ and $j<k \leq 10$.
\item[\underline{3folds}:]  The secant variety $\sigma_2$. \\
					The cone $J(L_{\{i,j\}}, C)$ where $1\leq i<j\leq 3$.
\item[\underline{4folds}:]  The cone $J(p_1, \sigma_2)$.
\end{itemize}

\vspace{3mm}
For each special effect variety, we compute its contribution to the dimension of $\mathcal L'$, which is given by \eqref{contribution of join}. For that, we take into account the exceeding number of $\mathcal L'$, defined in Definition \ref{def-epsilon}. In this case $\varepsilon =1$.
For instance, the contribution of $J(L_{\{2\}}, \sigma_1)$ to the dimension of $\mathcal L'$ is 
$$ F_1(5+3-2-1, 10, 1,5) = - \binom{5}{5} + 2\binom{6}{5} - \binom{1}{1} \binom{5}{4} = 8.  $$
We can compute this number with the following function of the Magma library
\begin{tbox}
\begin{verbatim} 
> s:=10; // number of base points
> e:= 1; // exceeding number
> I:={2};
> t:=1;
> r:=del(I,t); // computes the dimension of the cycle
> k:=ka(n,d,m,I,t); // computes the associated parameter k_I,t
> Fu(n+k-r-1,s,e,n,t); // computes the contribution 
8
\end{verbatim}
\end{tbox}

Taking all these contributions into account with their corresponding sign, as in Theorem \ref{ThmC}, we obtain the dimension of $\mathcal L'$. This is done by the following function of the Magma library.
\begin{tbox}
\begin{verbatim} 
> Dim(n,d,m);
6
\end{verbatim}
\end{tbox}

Thus, $\dim(\mathcal L)=\dim(\mathcal L') = 6$.

\section{Proof of the main theorem}\label{ProofMainTheorem}	

The rest of this manuscript is exclusively devoted to the proof of Theorem \ref{ThmC}. 
The proof is based on a backward induction on the exceeding number $\varepsilon$ and on a double induction on the dimension of the ambient space $n$ and on the integer $k_C^+$. It will be broken down is several steps, according to the following plan.
\begin{itemize}
\item We show that theorem holds for non-redundant linear systems with $n=2$ (Proposition \ref{ThmdimP2}). 
\item Using induction on $n$, we prove that the theorem is true when $k_C^+=0$. 
(Propostition \ref{propB}). 
\item We use induction on $n$ and on $k_C^+$ simultaneously to cover the case $\varepsilon=s-n-3$
(Proposition \ref{propC}). 
\item We settle the homogeneous case $m_1=\ldots = m_s = k_C$ (Proposition \ref{propD}). 
\item We show that Theorem \ref{dim}  holds in general by induction using all of the above cases as base steps.
\end{itemize}



We now state five  arithmetical properties of the function $F_t(a,s,\varepsilon,n)$ that we will extensively  use in the proof of Theorem \ref{dim}. The proof of these properties will be given in  Appendix \ref{Fiveproperties}.

%
\begin{property} \label{property2} 
$F_t(a,s-1,\varepsilon, n)  + F_{t-1}(a+1, s, \varepsilon, n)=F_t(a,s, \varepsilon, n)$, for every $t, s, \varepsilon, n \in \mathbb N$ with $t\geq 1$, and every $a \in \mathbb Z$.
\end{property}
\begin{property}\label{property4}  
$F_t(a,s, 0, n) + F_{t-1}(a, s, 0, n-1) = F_t(a+t, s, s-n-3, n)$, for every $t, a, s, n \in \mathbb N$, with $t\geq 1$ and $s\geq n+3$.\end{property}
\begin{property} \label{property5} \label{vareps02} 
$F_t(a,s, 0, n) + F_{t}(a, s, 0, n-1) = F_t(a+t+1, s, s-n-3, n),$ for every $t, a, s, n \in \mathbb N$ with $s\geq n+3$.
\end{property}
\begin{property} \label{property6} \label{vareps2}
$F_t(a,s, \varepsilon, n) + F_{t-1}(a,s, \varepsilon, n-1) = F_t(a, s, \varepsilon -1,n)$, for all $t, a, s, \varepsilon, n \in \mathbb{N}$ with  $\varepsilon \neq 0$.
\end{property}
\begin{property} \label{property7} \label{vareps1}
$F_t(a,s, \varepsilon, n) + F_t(a, s, \varepsilon, n-1) = F_t(a+1, s, \varepsilon -1, n)$, for every $t, a, s, \varepsilon, n \in \mathbb{N}$ with $\varepsilon \neq 0 $.
\end{property}


\subsection{The planar case}
 The dimensionality problem for linear systems of plane curves with multiple points on a smooth conic was solved algorithmically in the  language of Hilbert functions of schemes of fat points in \cite[ Theorem 3.1]{C}.
 We will give an alternative geometric proof here, that stresses how lines though pairs of points and the conic contribute to the dimension formula. 
 

In the notation of Section \ref{preliminaries}, given a curve class $D=dH-\sum_{i=1}^s m_i E_i$ on $X_s^2$, with $m_i >0$ for every $i \in \{1, \ldots, s\}$, we set 
\begin{align} \label{G(D)}
G(D) := &\binom{d+2}{2} - \sum_{i=1}^s\binom{m_i+1}{2} + \sum_{1\leq i<j\leq s} \binom{k_{ij}}{2} \\
				       & + \left(\binom{k_C}{2} + (s-5)\binom{k_C+1}{2} - \varepsilon \binom{k_C}{1} \right) . \nonumber
\end{align}
We will show that if $D$ is a non-redundant divisor, then 
$ h^0(X_s^2, D) = G(D).$




Before doing that, we recall the description of the effective cone of $X^2_s$ via its extremal rays, that are the rays spanned by
 the exceptional divisors and by the strict transforms of the lines $L_{ij}$ and of the conic $C$.
\begin{theorem} \label{EffP2} (\cite{CT}, Theorem 1.2)   
Let $X_s^2$ be the blow-up of $\mathbb{P}^2$ along $s$ points on a conic $C$. The effective cone of $X^2_s$ is finitely generated by the classes
\begin{enumerate}
\item $E_i$, for every $1\le i \le s$,
\item
$H - E_i - E_j$, for every $1\leq i < j \leq s$,
\item $2H - \sum_{i=1}^s E_i. $
\end{enumerate}
\end{theorem}
Recall that divisors in $\textrm{Pic}(X^2_s)$ satisfy the following intersection table
\begin{equation} \label{intersectiontable}
H^2= 1,  \quad \quad H \cdot E_i=0, \quad E_i \cdot E_j  = -\delta_{ij},
\end{equation}
for every $ i,j \in \{1, \ldots, s\}$.
Since the nef cone of $X^2_s$ is dual to the effective cone with respect to \eqref{intersectiontable}, as an immediate corollary of Theorem \ref{EffP2}  we shall list the defining inequalities of the former in the coordinates $d,m_1,\dots,m_s$ of the N\'eron-Severi space. 
\begin{corollary} \label{NefCone}
Let $X_s^2$ be the blow-up of $\mathbb{P}^2$ along $s$ points on a conic $C$ and let $D=dH - \sum_{i=1}^s m_i E_i$ be any divisor on $X$. Then $D$ is nef if and only if it satisfies the following inequalities
\begin{enumerate}
\item $m_i\ge 0$, for every $1\le i \le s$,
\item $d-m_i-m_j\geq 0$,  for $1\leq i < j \leq s$,
\item $2d - \sum_{i=1}^s m_i \geq 0$.
\end{enumerate}
\end{corollary}
 We will now show that nef divisors on $X^2_s$ have vanishing first cohomology group, as a consequence of the following lemma.
\begin{lemma} \label{HarLemma} (\cite{Har2}, Lemma II.7)
Let $X$ be a smooth projective rational surface and let $\mathcal{N}$ be the class of an effective divisor without fixed components. If $\mathcal{N}\cdot K_X < 0$, then $h^1(X, \mathcal N)=0$, where $K_X$ denotes the canonical divisor of $X$.
\end{lemma}

\begin{proposition} \label{Vanishingh1}
If $D$ is a nef divisor on $X_s^2$, then $h^1(X_s^2,D)=0$.
\end{proposition}
\begin{proof} Consider a nef divisor $D=dH-\sum m_i E_i$  and
recall that the canonical divisor of $X_s^2$ is
$ K_{X_s^2} = -3H + \sum E_i $.
Using \eqref{intersectiontable},  we obtain
$$ D \cdot K_{X_s^2}= \sum_{i=1}^s m_i -3d \leq - d < 0,$$
where the first inequality follows from Corollary \ref{NefCone}.
By Lemma \ref{HarLemma}, we conclude that $h^1(X_s^2,D)=0$.
\end{proof}
\begin{remark} \label{G(nefD)}
Notice that if $D$ is nef, then $k_C(D) \leq 0$ and $k_{ij}(D) \leq 0$, for every $1\leq i < j \leq s$. Therefore $G(D)$ as defined in (\ref{G(D)}) reduces to
$$G(D) = \binom{d+2}{2} - \sum_{i=1}^s \binom{m_i+1}{2} =  \text{\Large{$\chi$}}(X_s^2,D). $$
Therefore, by Proposition \ref{Vanishingh1},  for nef divisors we have $h^0(X_s^2, D) = G(D)$, so that Theorem \ref{ThmC} holds in this case.
\end{remark}
We will now  show that Theorem \ref{ThmC} holds for non-redundant effective and not necessarily nef divisors, by reducing any such divisor to a nef one, and showing that the formula $G(D)$ is preserved by this reduction.

\begin{lemma} \label{lemmakij}
Let $D=dH - \sum_{i=1}^s m_i E_i$ be an effective divisor on $X_s^2$, with $m_i \geq 0$ for every $i \in \{1, \ldots, s\}$. Then, for any $1\leq j_1 < j_2 \leq s$,
$$D' = D - k_{j_1j_2}(D)^+ (H- E_{j_1} - E_{j_2}), $$
is an effective divisor with $m_i' \geq 0$, for every $i \in \{1, \ldots, s\}$, and such that
\begin{align*}
\text{(a)} & \;\; k_{j_1j_2}(D')  = 0,  \\[2mm]
\text{(b)} & \;\; k_{i_1 i_2}(D')^+ = k_{i_1 i_2}(D)^+  \; \; \text{for any $1\leq i_1 < i_2 \leq s$, such that $(i_1, i_2) \neq (j_1, j_2)$},   \\[2mm]
\text{(c)} &\;\; k_C(D') = k_C(D).
\end{align*}
In addition, we have that
$$G(D') = G(D). $$
\end{lemma}
Before proving the lemma, we recall the following identity that we will make use of extensively: for $a,b\in\mathbb{N}$, we have
\begin{equation} \label{TriangularNumbers}
\binom{a + b+1}{2} = \binom{a+1}{2}+ \binom{b+1}{2} +ab.
\end{equation}

\begin{proof}[Proof of Lemma \ref{lemmakij}] After reordering the indices if needed, we may assume that $(j_1, j_2) = (1,2)$. We first notice that if $k_{12}(D)^+=0$, then $D'=D$ so that the lemma follows immediately. Thus, let us assume that $k_{12}(D)>0$ and, for simplicity, set $k_{12} = k_{12}(D)$.
We recall that $k_{12}$ corresponds to a lower bound for the multiplicity of containment of the line $L_{12}$ in the base locus of $|D|$. This implies that $H- E_1 - E_2$ is a fixed component of $|D|$ and, in particular, that
\begin{align*}
 D' & = D - k_{12}(H-E_1 - E_2) \\
	& = (d-k_{12}) H - (m_1 - k_{12})E_1 - (m_2 - k_{12}) E_2 - \sum_{i=3}^s m_i E_i 
\end{align*}
is still effective.

Now, for $i\in \{3, \ldots, s\}$ we have that $m_i ' = m_i \geq 0$ by hypothesis. If $i=1$,
$$m_1' = m_1 - k_{12} = m_1 - (m_1 + m_2 -d) = d-m_2 \geq 0, $$
by the effectivity of $D$. Similarly, we get $m_2' = d-m_1 \geq 0$. 
Moreover, it is straightforward to check that
\begin{itemize}
\item[(a)] $k_{12}(D') = (m_1 - k_{12}) + (m_2 - k_{12}) - (d-k_{12}) = m_1 + m_2 - d - k_{12} = 0. $

\item[(b)] For $i \in \{1,2\}$ and $j\in \{3, \ldots, s\}$, we have that
$$ k_{ij}(D') = (m_i - k_{12}) + m_j - (d-k_{12}) = m_i + m_j - d = k_{ij}.$$
Otherwise, if $i, j \in \{3, \ldots, s\}$, we have that
$$ k_{ij}(D') = m_i + m_j - (d-k_{12}) = m_1 + m_2 + m_i + m_j - 2d \leq 0, $$
where the last inequality holds because $D$ is effective.
Since we are assuming that $k_{12} >0$, then we must have $$k_{ij}(D) = m_i + m_j -d \leq 0. $$
Thus, we have that $k_{ij}(D')^+= k_{ij}(D)^+ =0$ in this case.

\item[(c)] Notice that
\begin{align*} \sum_{i=1}^2 m_i' - 2d' &= (m_1 - k_{12}) + (m_2 - k_{12}) + \sum_{i=3}^s m_i - 2(d-k_{12}) \\
								& = \sum_{i=1}^s m_i - 2d,
\end{align*}
which implies that $k_C(D') = k_C(D)$.
\end{itemize}
In addition, applying (\ref{TriangularNumbers}) we have that
 $$ \binom{d-k_{12} +2}{2} = \binom{d+2}{2}  - \binom{k_{12} + 1}{2} - (d-k_{12}+1)k_{12}, $$
and
$$\binom{m_i - k_{12}+1}{2} = \binom{m_i +1}{2} - \binom{k_{12} +1}{2} - (m_i - k_{12})k_{12}. $$
Putting everything together, we obtain that
$$ \binom{d-k_{12} +2}{2}  - \binom{m_1 - k_{12}+1}{2} - \binom{m_2 - k_{12} +1}{2} $$
$$ = \binom{d+2}{2} - \binom{m_1 +1}{2} - \binom{m_2 +1}{2} + \left( \binom{k_{12} +1}{2} - k_{12} \right) +  \left((m_1 + m_2-d)k_{12} -k_{12}^2 \right)$$
$$ =  \binom{d+2}{2} - \binom{m_1 +1}{2} - \binom{m_2 +1}{2} +  \binom{k_{12}}{2} + 0. $$
This, together with (a), (b) and (c), implies that
$G(D') = G(D).$
\end{proof}

\begin{lemma} \label{lemmakC}
Let $D=dH - \sum_{i=1}^s m_i E_i$ be a non-redundant effective divisor on $X_s^2$. Then, 
$$D' = D - k_C(D)^+ (2H- \sum_{i=1}^s E_i), $$
is an effective divisor with $m_i' \geq 0$, for every $i \in \{1, \ldots, s\}$, and such that
\begin{align*}
\text{(a)} & \;\; k_{ij}(D') = k_{ij}(D) \; \; \text{for every $1\leq i < j \leq s$},   \\[2mm]
\text{(b)} &\;\; k_C(D') \leq 0. 
\end{align*}
In addition, we have that
$$G(D') = G(D). $$
\end{lemma}
\begin{proof} If $k_C(D)^+ = 0$, then $D'=D$ and the lemma follows. Thus, let us assume that $k_C(D) >0$ and, for simplicity let us set $k_C = k_C(D)$. Take 
$$
D' = D - k_C (2H - \sum_{i=1}^s m_i E_i)
	 = (d-2k_C)H - \sum_{i=1}^s (m_i - k_C) E_i,
$$
which is effective since $k_C$ is a lower bound for the multiplicity of containment of $C$ in the base locus of $|D|$. Moreover, the non-redundancy of $D$ implies that
$ m_i' = m_i - k_C \geq 0$, for every $i \in \{1, \ldots, s\}$.
It is simple to see that
\begin{itemize}
\item[(a)] For every $1 \leq i < j \leq s$,
$$k_{ij}(D') = (m_i - k_C) + (m_j - k_C) - (d-2k_C) = m_i + m_j -d = k_{ij}(D).$$
\item[(b)] Let us assume without loss of generality that $m_1 \geq \ldots m_r >k_C$ and $m_{r+1} = \ldots =m_s = k_C$. By definition, we have that, if $\varepsilon = \varepsilon(D)$ is the exceeding number of $D$, then
$$ k_C = k_C(D) = \frac{ \sum_{i=1}^r m_i +(s-r)k_C - 2d + \varepsilon}{s-4}$$
which implies that
$$ \sum_{i=1}^r m_i -2d = (r-4)k_C -\varepsilon $$
that is,
$$ \sum_{i=1}^r (m_i-k_C) -2d = -4k_C -\varepsilon <0. $$
Hence, 
$$k_C(D') = \left\lceil \frac{\sum_{i=1}^r(m_i -k_C) -2d}{r-4} \right\rceil \leq 0. $$
\end{itemize}
Finally, we observe that, by applying (\ref{TriangularNumbers}) twice, we have that
\begin{align*}
 \binom{d-2k_C+2}{2} &= \binom{d+2}{2} - \binom{2k_C +1}{2} - (d-2k_C +1)2k_C \\
				& = \binom{d+2}{2} - \left( 2\binom{k_C+1}{2} + k_C^2 \right) - (d-2k_C+1)2k_C
\end{align*}
and also, for every $i \in \{1, \ldots, s\}$
$$ \binom{m_i -k_C + 1}{2} = \binom{m_i +1}{2} - \binom{k_C +1}{2} - (m_i -k_C)k_C. $$
Thus,
$$ \binom{d-2k_C+2}{2} - \sum_{i=1}^s \binom{m_i -k_C +1}{2} $$
\begin{equation} \label{-kcC} = \binom{d+2}{2} - \sum_{i=1}^s \binom{m_i+1}{2} + (s-2) \binom{k_C +1}{2} -k_C^2 -2k_C \end{equation}
$$ +\left (\sum_{i=1}^s m_i -2d\right)k_C -(s-4)k_C^2.$$
On the one hand, we have
$$ \sum_{i=1}^s m_i -2d = (s-4)k_C -\varepsilon, $$
where $\varepsilon$ is the exceeding number of $D$, 
which implies that
$$ \left (\sum_{i=1}^s m_i -2d\right)k_C -(s-4)k_C^2 = - \varepsilon k_C = -\varepsilon \binom{k_C}{1}.$$
On the other hand, we have
\begin{align*}
(s-2) \binom{k_C+1}{2} - k_C^2 -2k_C & = (s-5)\binom{k_C+1}{2} + \left( 3\binom{k_C+1}{2} - k_C^2 -2k_C \right) \\
							&= (s-5)\binom{k_C +1}{2} + \binom{k_C}{2}.
\end{align*}
Hence, equation (\ref{-kcC}) reads as follows
\begin{align*}
 \binom{d-2k_C+2}{2} - \sum_{i=1}^s \binom{m_i -k_C +1}{2} = & \binom{d+2}{2} - \sum_{i=1}^s \binom{m_i+1}{2} \\
			& \; +\binom{k_C}{2} + (s-5)\binom{k_C+1}{2} - \varepsilon \binom{k_C}{1} 
\end{align*}
This, together with parts (a) and (b) of the lemma, implies that
$ G(D') = G(D). $
\end{proof}

\begin{propA} \label{ThmdimP2}
Let $D=dH-\sum_{i=1}^s m_i E_i$ be a non-redundant effective divisor on $X_s^2$, the blow-up of $\mathbb{P}^2$ at $s\geq 5$ points on a smooth conic $C$. Then 
$h^0(X_s^2, D)= G(D)$.
\end{propA}
\begin{proof}
Let $D$ be a non-redundant effective divisor for which we also assume that $m_i \geq 0$ for every $i\in \{1, \ldots, s\}$. By recursively applying Lemma \ref{lemmakij} and Lemma \ref{lemmakC} we have that the divisor
$$ D' = D - \sum_{1\leq i < j \leq s} k_{ij}(D)^+(H-E_i - E_j) - k_C(D)^+(2H- \sum_{i_1}^s E_i) $$
is effective and it satisfies that
$m_i' \geq 0$, $k_{ij}(D') = m_i' + m_j' -d' \leq 0$ for every   $1\leq 1 < j \leq s$, and $k_C(D')  \leq 0$ (which implies $\sum m_i' -2d' \leq 0$). Thus $D'$ is nef by Corollary \ref{NefCone}. In addition, Lemma \ref{lemmakij} and Lemma \ref{lemmakC} imply that $G(D')=G(D)$.

We recall that, $k_{ij}(D)^+$ and $k_C(D)^+$ are lower bounds for the multiplicity of containment of $L_{ij}$ and $C$, respectively, in the base locus of $|D|$. Therefore, $D'$ is the result of removing from $D$ some fixed components, which means that 
$$ h^0(X_s^2, D) = h^0(X_s^2, D'). $$
To conclude, notice that, by Remark \ref{G(nefD)}, nefness of $D'$ implies that $h^0(X_s^2, D') = G(D')$. Altogether, we have that
$$  h^0(X_s^2, D) = h^0(X_s^2, D') = G(D') = G(D), $$
which concludes the proof.
\end{proof}

\subsection{The case $k_C^+=0$}\label{kC+=0}

In the notation of Section \ref{preliminaries}, given a divisor  $D = dH-\sum_{i=1}^sm_iE_i$ on $X^n_s$, we set 
\begin{equation}\label{ldim}
\ldim(D):=\sum_{I\subseteq \{1, \ldots, s\} } (-1)^{|I|} \binom{n+k_I(D)- |I|}{n}.
\end{equation}
We shall show that if $D$ is effective and non-redundant, with $k_C(D)^+=0$, then the dimension of its space of global sections is computed in formula \ref{ldim}.
\begin{remark}
Notice that $d\geq m_i$, which holds for effective divisors, and $\sum_{i=1}^sm_i-nd\le0$, which holds under the assumption $k_C(D)^+=0$, imply that for every $I \subset \{1, \ldots, s\}$,
$$ k_{I, \sigma_t} = \sum_{i\in I}m_i + tk_C - (|I| + t -1)d \leq 0,$$
 for any special cycle $J(I, \sigma_t)$ with $t\geq 1$. Therefore all terms involving $k_{I, \sigma_t}$ in the dimension formula \eqref{dim} vanish and in particular \eqref{dim} and \eqref{ldim} coincide. 
 In other terms, proving that the main theorem, Theorem \ref{ThmC}, holds for divisors with $k_C^+=0$ is equivalent to showing that they are \emph{linearly non-special}, according to the definition of \cite[Definition 3.2]{BDPLinear} generalised to the case of an arbitrary number of points on $C$.  \end{remark}

\begin{propA} \label{propB}
Let $D=dH-\sum_{i=1}^s m_i E_i$ be an effective divisor on $X_s^n$, such that $k_C(D)^+=0$. Then 
$H^0(X^n_s,D)=\ldim(D)$. 
\end{propA}
Before proving the proposition, we recall the following identity sometimes referred to as the \emph{hockey-stick identity}, that will be used in the proof. For any  $a, b\in\mathbb{N}$, we have 
\begin{equation} \label{hockey}  \sum_{\lambda=0}^b \binom{a+\lambda}{a} = \binom{a + b + 1}{a+1}. \end{equation}

\begin{proof}[Proof of Proposition \ref{propB}]
 The proof will be by induction on $s$, with base case $s=n+2$, and on $n$, with base case $n=2$. 

 If $s\le n+2$, then the hypothesis $k^+_C=0$ is automatically satisfied by any effective divisor. Moreover that $h^0(X^n_s,D)=\ldim(D)$ was proved in \cite[Corollary 4.8]{BDPLinear}. 
If $n=2$ and $k_C(D)^+=0$, then $h^0(X^n_s,D)=\ldim(D)$ by Proposition \ref{ThmdimP2}.

 Using Castravet and Tevelev's exact sequence that resulted in \eqref{indstep} and the assumption that $k^+_C=0$, we obtain
$$ h^0(X_s^n,D)= h^0(X_s^n,D+E_1) - h^0(X_{s-1}^{n-1}, l(D+E_1)), $$
where $l$ is the linear map on $\text{Pic}(X_s^n)$ defined in (\ref{lmap}).
 After recursively applying the latter, and noticing that $k_C(D+\mu E_1)^+=0$, for $\mu\ge 1$, we arrive to 
 \begin{equation} \label{formula}
h^0(X_s^n,D)= h^0(X_s^n,D+m_1 E_1) - \sum_{\mu=1}^{m_1} h^0(X_{s-1}^{n-1}, l(D+ \mu E_1)).  
\end{equation}

We first compute the first summand of the right hand side of \eqref{formula}.
Notice that $D+m_1E_1 = dH - \sum_{i=2}^s m_i E_i$ is effective on $X_{s-1}^n$. Since by   $k_C(D)^+=0$, that gives  $\sum_{i=1}^s m_i - nd\le0$, then $  \sum_{i=2}^s m_i - nd\le 0, $ and so
$$ k_C(D+m_1 E_1) = \left\lceil \frac{ \sum_{i=2}^s m_i -nd}{(s-1)-n-2} \right\rceil \leq 0. $$
 Hence, by the induction hypothesis on the number of points $s$, we know that $H^0(X_{s-1}^n, D+m_1E_1)$ is linearly non-special, which means that
\begin{equation} \label{s-1induction} h^0(X_s^n,D+m_1 E_1)= \ldim(D+m_1 E_1)=\sum_{1 \notin I} (-1)^{|I|} \binom{n + k_I - |I|}{n} , \end{equation}
 where we use the notation $1 \notin I$ to indicate that the sum ranges over all the index sets $I\subset \{ 2, \ldots, s \}$, including the empty set $I=\emptyset$. In this last case, we keep using the convention $k_{\emptyset}=d$.

Now we  compute the remaining summands of (\ref{formula}).
Recall that $X_{s-1}^{n-1}$ denotes the blow-up of $\mathbb P^{n-1}$ at points $q_2, \ldots, q_s$ that sit on a rational normal curve of degree $n-1$ that  is denoted with $C_{n-1}$. 
For every  $1\leq \mu \leq m_1$, we set $D_{n-1}^\mu := l(D+\mu E_1)$ and we can write
$$D_{n-1}^\mu = d_{n-1}^\mu H_{n-1} - \sum_{i=2}^s m_{i, n-1}^\mu  E_{i,n-1}  \in \text{Pic}(X_{s-1}^{n-1})$$
where
 $d_{n-1}^\mu := m_1 - \mu$ and 
$ m_{i, n-1}^\mu := m_1 + m_i - \mu - d$, for every $i\in \{2, \ldots, s\}$. 
Moreover $k_{C_{n-1}}(D_{n-1}^\mu)^+ =0$ since
\begin{align*}
\sum_{i=2}^{s} m_{i,n-1}^\mu - (n-1)d_{n-1}^\mu &= \sum_{i=2}^{s}m_i + (s-1)(m_1 -\mu -d) -(n-1)(m_1-\mu) \\
										&= \left(\sum_{i=1}^s m_i - nd \right) + (s-n-1) \left(m_1 - \mu -d\right) - \mu,
\end{align*} 
which is nonpositive because $\sum_{i=1}^s m_i - nd\le0$ by the assumption $k_C^+(D)=0$
 and because
$ (s-n-1)(m_1 - \mu - d) \leq 0, $
 since $s\geq n+3$ and since $m_1 \leq d$, the latter following the effectivity assumption on $D$. 

 Therefore, by the induction hypothesis on the dimension $n$, we can say that the linear system corresponding to $D_{n-1}^\mu$ is linearly non-special so that
 \begin{equation}\label{D-mu-n-1}
 h^0(X_{s-1}^{n-1}, D_{n-1}^\mu) =\ldim(D_{n-1}^\mu)=\sum_{I \subseteq \{2, \ldots, s\}} (-1)^{|I|} \binom{n-1 + {k}_{I,n-1}^\mu - |I|}{n-1},
 \end{equation}
 where we set $k_{I,n-1}^\mu:= k_{I}(D^\mu_{n-1})$.
 We will now reformulate the integer \eqref{D-mu-n-1} in such a way that it  depends explicitly on the parameters $k_I(D)$.
 In order to do so, for every subset $I \subseteq \{2, \ldots, s\}$, we  compute
\begin{align*}
k_{I,n-1}^\mu&= \sum_{i\in I} m_{i,n-1}^\mu - (|I|-1) d_{n-1}^\mu \\
			     &= \sum_{i \in I} m_i + |I|(m_1 - \mu - d)  - (|I|-1)( m_1 - \mu) \\
			& = m_1 + \sum_{i\in I} m_i - |I|d - \mu = k_{I\cup \{1\}}(D) - \mu.
\end{align*}
Substituting in \eqref{D-mu-n-1}, we obtain
\begin{align*}
 h^0(X_{s-1}^{n-1}, D_{n-1}^\mu)	&= \sum_{I \subseteq \{2, \ldots , s \}} (-1)^{|I|} \binom{ n + k_{I \cup \{1\}} - \mu - (|I|+1)}{n-1} \\
						&=  \sum_{1 \in I \subseteq \{1, \ldots , s \}} (-1)^{(|I|-1)} \binom{ n + k_I - \mu - |I|}{n-1} \label{hockeyapp}.
\end{align*}
 Hence,
\begin{align*} 
\sum_{\mu = 1}^{m_1}  h^0(X_{s-1}^{n-1}, D_{n-1}^\mu) &=  \sum_{\mu=1}^{m_1} \sum_{1 \in I \subseteq \{1, \ldots , s \}} (-1)^{(|I|-1)} \binom{ n + k_I - \mu - |I|}{n-1} \nonumber \\
		&= -  \sum_{1 \in I \subseteq \{1, \ldots , s \}} (-1)^{|I|} \sum_{\mu=1}^{m_1} \binom{ n + k_I - \mu - |I|}{n-1} \nonumber \\
		&= -  \sum_{1 \in I \subseteq \{1, \ldots , s \}} (-1)^{|I|} \sum_{\lambda = k_I - m_1 - |I| + 1}^{k_I - |I|} \binom{ n -1 + \lambda}{n-1} \nonumber \\
		&= -  \sum_{1 \in I \subseteq \{1, \ldots , s \}} (-1)^{|I|} \sum_{\lambda = 0}^{k_I - |I|} \binom{ n -1 + \lambda}{n-1}.		
\end{align*}
Observing that since $I\ni 1$, then $m_1 \geq k_I$ because $D$ is effective, we obtain $ k_I - m_1 - |I| -1 \leq -|I| + 1 \leq 0, $ and so
$$\sum_{\mu = 1}^{m_1}  h^0(X_{s-1}^{n-1}, D_{n-1}^\mu) = -  \sum_{1 \in I \subseteq \{1, \ldots , s \}} (-1)^{|I|} \sum_{\lambda = 0}^{k_I - |I|} \binom{ n -1 + \lambda}{n-1}.$$
 Applying the identity \eqref{hockey} to the above, we get 
\begin{equation} \label{1inIcase} \sum_{\mu = 1}^{m_1}  h^0(X_{s-1}^{n-1}, D_{n-1}^\mu) = - \sum_{1\in I} (-1)^{|I|} \binom{n+k_I- |I|}{n},\end{equation}
 where $1\in I$ indicates that the sum ranges over all the index sets $I\subseteq \{1, \ldots, s\}$ such that $1\in I$. 

 Finally, we can substitute (\ref{s-1induction}) and (\ref{1inIcase}) in the equation (\ref{formula}):
\begin{align*}
h^0(X_s^n,D) &= h^0(X_s^n,D+m_1 E_1) - \sum_{\mu=1}^{m_1} h^0(X_{s-1}^{n-1}, l(D+ \mu E_1)) \\
			&=   \sum_{1 \notin I} (-1)^{|I|} \binom{n + k_I - |I|}{n} +  \sum_{1\in I} (-1)^{|I|} \binom{n+k_I- |I|}{n} \\
			& =  \sum_{I\subseteq \{1, \ldots, s\} } (-1)^{|I|} \binom{n+k_I- |I|}{n},
\end{align*}
concluding the proof.
\end{proof}


\subsection{The case $\varepsilon = s-n-3$}\label{epsilon0s-n-s}
In this section, using the same notation of Section \ref{preliminaries}, we 
consider the case where the exceeding number   $\varepsilon=\varepsilon(D)$ (cf. Definition \ref{def-epsilon}) attains its maximum value, that is $\varepsilon = s-n-3$. We shall prove that the Main Theorem \ref{ThmC} holds under this assumption. Since the case $k_C^+(D)=0$ is already covered in Section \ref{kC+=0}, we  can make the further assumption that $k_C \geq 1$.

\begin{propA} \label{propC}
Let $D=dH-\sum_{i=1}^s$ be a non-redundant effective divisor on $X^n_s$. Assume that $\varepsilon(D) = s-n-3$ and that $k_C(D) \geq 1$. Then 
\begin{equation} \label{dim-epsilon=s-n-s} h^0(X_s^n, D)= \sum_{I, \sigma_t} (-1)^{|I|}F_t (n + k_{I, \sigma_t} - r_{I, \sigma_t} - 1, s, s-n-3, n). \end{equation}
\end{propA}
\begin{proof}
The proof is by induction on the integers $n$ and on $k_C=k_C(D)$. For every pair $(n,k_C)$, we will assume the statement true for 
\begin{itemize}
\item all effective non-redundant divisors with parameters $(n,k_C-1)$ and any exceeding number (the base step is covered in Proposition \ref{propB}) and 
\item for all effective non-redundant divisors $(n-1,k_C)$ and any exceeding number (with base case covered in Proposition \ref{ThmdimP2}).
\end{itemize}
The inductive step is provided by equation \eqref{indstep}: 
\begin{equation} \label{excnumfirststep} h^0(X_s^n,D)= h^0(X_s^n,D+E_1) - h^0(X_{s}^{n-1}, l(D+E_1)-k_C^+(D+E_1) E_q).\end{equation}
We compute all summands appearing on the right hand side of (\ref{excnumfirststep}) individually.

 First, we consider the divisor $D+E_1$. 
Assume first of all that $m_1=1$. Since by the hypotheses $m_1\ge k_C\ge 1$  then $k_C=1$ and, in particular, $\sum_{i=2}^sm_i-nd=0$, due to the fact that $\varepsilon = s-n-3$.
This implies that
$$ k_C(D+E_1) = \left\lceil \frac{ \sum_{i=2}^s m_i - nd}{(s-1)-n-2} \right \rceil = 0. $$
 Now, assume that $m_1 \geq 2$.
 Since $\varepsilon(D)= s-n-3$, it is straightforward to check that 
 \begin{align*}
 k_C(D+E_1) &= \left \lceil \frac{m_1 - 1 + \sum_{i=2}^s m_i - nd}{s-n-2} \right \rceil \nonumber \\
		& = \left \lceil \frac{\sum_{i=1}^s m_i - nd + (s-n-3) - (s-n-2)}{s-n-2} \right\rceil  \nonumber \\
		& =  \frac{\sum_{i=1}^s m_i - nd + (s-n-3)}{s-n-2} - \frac{ s-n-2}{s-n-2}  \nonumber \\ 
		& =k_C  -1. \label{kcde1epscase}
\end{align*}
In both cases we conclude that $k_C(D+E_1) = k_C - 1$ and that $D+E_1$ has exceeding number $\varepsilon(D+E_1)=0$.
Moreover, since $D$ is non-redundant, then so is $D+E_1$. In fact we have $m_1 - 1 = k_C -1 \geq  k_C(D+E_1)$ and
$m_i \geq  k_C\ge k_C(D+E_1)$ for any $i \in \{2, \ldots, s\}$. 
Then, by induction on the integer $k_C$, the dimension of $H^0(X^n_s,D+E_1)$ is 
\begin{equation} \label{dimension1}
h^0(X_s^n, D+E_1)= \sum_{\substack{ 0 \leq t \leq n/2 \\ 0\leq |I| \leq n-2t  }} (-1)^{|I|}F_t (n + k_{I, \sigma_t}(D+E_1) - r_{I, \sigma_t} - 1, s, 0, n).
\end{equation}
In order to obtain an explicit expression, we write each integer $k_{I,\sigma_t}(D+E_1)$ in terms of the integer $k_{I,\sigma_t}=k_{I,\sigma_t}(D)$.
Using the relation $k_C(D+E_1) = k_C-1$, a simple calculation shows that
\[
k_{I, \sigma_t}(D+E_1)=
	\begin{cases}
		k_{I, \sigma_t}-t &\quad \text{if } 1 \notin I \\
		k_{I, \sigma_t}-t-1 & \quad \text{if } 1 \in I.
	\end{cases}
\]
 We can plug this in formula (\ref{dimension1}) to obtain
\begin{align}
h^0(X_s^n,D+E_1) 	 = &\sum_{\substack{ 0 \leq t \leq n/2 \\ 0\leq |I| \leq n-2t \\ 1 \notin I }}  (-1)^{|I|}F_t (n + k_{I, \sigma_t}-t - r_{I, \sigma_t} - 1, s, 0, n)  \nonumber \\
			    &+ \sum_{\substack{ 0 \leq t \leq (n-1)/2 \\ 0\leq |I| \leq n-2t \\ 1 \in I }}  (-1)^{|I|}F_t (n + k_{I, \sigma_t} -t -1 - r_{I, \sigma_t} - 1, s, 0, n). \label{d+e1caseeps2}
\end{align}
  We notice that the index $t$ in the second summation goes up to $t=(n-1)/2$ because if $t$ were $n/2$,  we would have
$ 0 \leq |I| \leq n - 2t = 0$ which is impossible.

We now consider the second summand of  the right hand side of (\ref{excnumfirststep}). To simplify notation we write $D_{n-1}:= l(D+E_1)-k_C^+(D+E_1) E_q$. It is an effective divisor  that we can expressed in the Picard group basis of $X_s^{n-1}$ as follows:
\begin{align*}
D_{n-1}&:= d_{n-1} H_{n-1}-\sum_{i=1}^s m_{i,n-1}E_{i,n-1}\\
 &\ = (m_1 - 1) H_{n-1} - (k_C -1 ) E_{1, n-1} - \sum_{i=2}^s (m_i + m_1 - 1 - d) E_{i,n-1}.
 \end{align*}
 Furthermore, letting $C_{n-1}$ denote the rational normal curve of degree $n-1$ passing through the blown-up points of $X_s^{n-1}$, we have that
\begin{align*} 
k_{C_{n-1}}(D_{n-1}) & = \left \lceil \frac{ k_C -1 + \sum_{i=2}^s m_i + (s-1)(m_1 - 1 -d) - (n-1)(m_1 -1)}{s-(n-1)-2} \right \rceil \\
				& =  \left \lceil \frac{ k_C + \sum_{i=1}^s m_i -nd  -2 + (s-n-1)(m_1 - 1) - (s-n-1)d}{s-n-1} \right \rceil .
\end{align*}
 Recalling that $\varepsilon(D)= s-n-3$, we get
$$  \sum_{i=1}^s m_i -nd =  (s-n-2)k_C - (s-n-3), $$
and so
\begin{align} \label{kcn-1eps}
k_{C_{n-1}}(D_{n-1}) & =  \left \lceil \frac{ k_C  + (s-n-2)k_C - (s-n-3) -2}{s-n-1} \right \rceil + m_1 - d -1 \nonumber  \\
				& = m_1 + k_C - d -1 - 1 \nonumber\\
				& = k_{\{1\}, \sigma_1} - 2.
\end{align}
This also shows that $\varepsilon(D_{n-1})=0$ and that $D_{n-1}$ is non-redundant. The latter holds because
$$ k_{C_{n-1}}(D_{n-1}) =  m_1 + k_C - d -2 < k_C - 1 = m_{1, n-1}, $$
where the inequality follows from the assumption $m_1\le d$, and because  for every $i \in \{2, \ldots, s\}$
$$ k_{C_{n-1}}(D_{n-1}) = m_1 + k_C - d -2 < m_1 + m_i -d -1 = m_{i, n-1}, $$
which follows from the fact that $D$ is non-redundant.
 Then, by induction on the dimension $n$,  the dimension of $H^0(X_s^{n-1}, D_{n-1})$ is 
\begin{equation} \label{dimension2}
h^0(X_s^{n-1}, D_{n-1})= \sum_{\substack{0\leq t \leq (n-1)/2 \\ 0\leq |I| \leq n-1-2t}} (-1)^{|I|}F_t (n-1 + k_{I, \sigma_t}(D_{n-1}) - r_{I, \sigma_t} - 1, s, 0, n-1). 
\end{equation}
In order to obtain an explicit expression, we write each integer  $k_{I,\sigma_t}(D_{n-1})$ in terms of $k_{I,\sigma_t}=k_{I,\sigma_t}(D)$. Using the relation
In (\ref{kcn-1eps}), one can easily verify that 
\begin{equation} \label{relationoverline}
k_{I, \sigma_t} (D_{n-1}) =
	\begin{cases}
		k_{I\cup \{1\}, \sigma_t}-t-1 &\quad \text{if } 1 \notin I \\
		k_{I\setminus \{1\}, \sigma_{t+1}}-t-1 & \quad \text{if } 1 \in I.
	\end{cases}
\end{equation}
We also have the following
\begin{equation} \label{relationoverline2}
r_{I, \sigma_t} =
	\begin{cases}
		r_{I\cup \{1\}, \sigma_t}-1 &\quad \text{if } 1 \notin I \\
		r_{I\setminus \{1\}, \sigma_{t+1}}-1 & \quad \text{if } 1 \in I.
	\end{cases}
\end{equation}
 We can plug (\ref{relationoverline}) and (\ref{relationoverline2}) into (\ref{dimension2}) to get
$$ h^0(X_s^{n-1}, D_{n-1}) $$
\begin{equation} \label{eqepsn-11}= \sum_{\substack{0 \leq t \leq (n-1)/2 \\ 0\leq |I| \leq n-2t \\ 1 \in I}}  (-1)^{|I|-1}F_{t-1} (n + k_{I, \sigma_t}-t - r_{I, \sigma_t} - 1, s, 0, n-1) \end{equation}
\begin{equation*}  + \sum_{\substack{1 \leq t \leq n/2 \\ 0\leq |I| \leq n-2t \\ 1 \notin I}}  (-1)^{|I|-1}F_t (n + k_{I, \sigma_t} -t -1 - r_{I, \sigma_t} - 1, s, 0, n-1). \end{equation*}

Finally we can compute $h^0(X_s^n, D)$ by putting together \eqref{excnumfirststep}, (\ref{eqepsn-11}) and (\ref{d+e1caseeps2}). Indeed, by separating the first summation of (\ref{d+e1caseeps2}) into two summations, one where $t=0$ and another where $1\leq t \leq n/2 $, we arrive to
\begin{align}
h^0(X_s^n,D) & = h^0(X_s^n, D+E_1) + h^0(X_s^{n-1}, D_{n-1}) \nonumber \\ 
		      & = A_0 + A_1 + A_2,  \label{alltogetherexcnumcase}
\end{align}
 where
$$ A_0	 =	 \sum_{\substack{ |I|=n \\ 1 \notin I }}  (-1)^{|I|}F_0 (n + k_{I, \sigma_t} - r_{I, \sigma_t} - 1, s, 0, n) \label{sum1}, $$
\begin{align*}
A_1      = &  \sum_{\substack{ 1\leq t \leq n/2 \\ 0\leq |I| \leq n-2t \\ 1 \notin I }}  (-1)^{|I|}F_t (n + k_{I, \sigma_t}-t - r_{I, \sigma_t} - 1, s, 0, n) \\
		&+ \sum_{\substack{ 1\leq t \leq n/2 \\ 0\leq |I| \leq n-2t \\ 1 \notin I }}    (-1)^{|I|}F_{t-1} (n + k_{I, \sigma_t}-t - r_{I, \sigma_t} - 1, s, 0, n-1),
\end{align*}
 and 
\begin{align*}
A_2      =&  \sum_{\substack{0 \leq t \leq (n-1)/2 \\ 0\leq |I| \leq n-2t \\ 1 \in I}}   (-1)^{|I|}F_t (n + k_{I, \sigma_t} -t -1 - r_{I, \sigma_t} - 1, s, 0, n) \\
		& +  \sum_{\substack{0 \leq t \leq (n-1)/2 \\ 0\leq |I| \leq n-2t \\ 1 \in I}}  (-1)^{|I|}F_t (n + k_{I, \sigma_t} -t -1 - r_{I, \sigma_t} - 1, s, 0, n-1). \label{sum3}
\end{align*}

We notice that, from the definition of the function $F_t$ in (\ref{F_t}), we have that $F_0$ is independent from the values of the parameter $\varepsilon$. Therefore, in $A_0$ we can make the following replacement,
$$ F_0 (n + k_{I, \sigma_t} - r_{I, \sigma_t} - 1, s, 0, n) = F_0 (n + k_{I, \sigma_t} - r_{I, \sigma_t} - 1, s, s-n-3, n). $$
Applying Property \ref{property4} to $A_1$ and Property \ref{property5} to $A_2$, we obtain
\begin{align*} A_1 &= \sum_{\substack{ 1\leq t \leq n/2 \\ 0\leq |I| \leq n-2t \\ 1 \notin I }}    (-1)^{|I|}F_t (n + k_{I, \sigma_t} - r_{I, \sigma_t} - 1, s, s-n-3, n),\\
A_2 &= \sum_{\substack{0 \leq t \leq (n-1)/2 \\ 0\leq |I| \leq n-2t \\ 1 \in I}}   (-1)^{|I|}F_t (n + k_{I, \sigma_t}  - r_{I, \sigma_t} - 1, s, s-n-3, n).  \end{align*}
Finally, substituting the latter expressions in (\ref{alltogetherexcnumcase}), we obtain the formula \eqref{dim-epsilon=s-n-s} and this concludes the proof.
\end{proof}


\subsection{The general case}
In this section  we prove that the Theorem \ref{ThmC} holds for $n\ge 3$ and for any non-redundant and non-empty linear system $\mathcal L = \mathcal L_{n,d}(m_1, \ldots, m_s)$ with   $k_C=k_C(D)\geq 1$ and with exceeding number $\varepsilon=\varepsilon(D) < s-n-3$. 

We first show the statement for a special class of homogeneous linear systems, that is such that $m_1= m_2 = \ldots = m_s = k_C$. 

\begin{propA} \label{propD}
Let $D=dH-m\sum_{i=1}^s E_i$ be an effective divisor on $X^n_s$ with $k_C(D)=m\ge1$. Then \eqref{dim} holds.
\end{propA}
\begin{proof}
If $\varepsilon = s-n-3$, then the statement follows from Proposition \ref{propC}. From now on, we assume that $\varepsilon < s-n-3$. 
The proof is by induction on  $n$, on $\varepsilon$ (backward), and on $s$, assuming the statement is true for:
\begin{itemize}
\item all effective non-redundant divisors in $X^{n-1}_s$, with base case $n=2$ covered in Proposition \ref{ThmdimP2};
\item all effective non-redundant divisors on $X^n_{s-1}$, where the base step, $s=n+2$, is covered in \cite[Corollary 4.8]{BDPLinear} and in this case $h^0(X^n_s,D)=\ldim(D)$, i.e. $D$ is linearly non-special;
\item all effective non-redundant divisors with exceeding number at least $\varepsilon+1$, with base case covered in Proposition \ref{propC}, where the exceeding number attains its maximum value.
\end{itemize}
Using the inductive step, described in \eqref{indstep}, we get
\begin{equation} \label{formula2indhomcase} h^0(X_s^n,D) =  h^0(X_s^n, D + E_1) - h^0(X_s^{n-1}, D_{n-1}), \end{equation}
 where we denote
$$D_{n-1} = l(D+E_1) - k_C^+(D+E_1)E_q.$$

We claim that $h^0(X_s^{n-1}, D_{n-1})=0$.  
To see this, assume first that $m\geq 2$:  using $\varepsilon < s-n-3$, we compute \begin{align*}
 k_C(D+E_1) & = \left\lceil \frac{(m -1) + \sum_{i=2}^s m - nd}{s-n-2} \right\rceil \nonumber \\
		& = \frac{\sum_{i=1}^s m - nd -1 + (\varepsilon + 1)}{s-n-2} \nonumber \\
		&= k_C(D). \label{kchomggencase}
\end{align*}
 If $m=1$, from the expression 
 $$ k_C = \frac{s -nd + \varepsilon}{s-n-2} = 1, $$
we obtain
$$ k_C(D+E_1) = \frac{(s-1) -nd + \varepsilon}{s-n-3} = 1. $$
In both cases, in the
expressing of $D_{n-1}$ in the Picard group basis of $X^{n-1}_s$, the coefficient of  $H_{n-1}$ is $m-1$ while the coefficient of $E_{1,n-1}$ is $- k_C(D+E_1)=-m$. The proof of the claim follows from the simple observation that  the only hypersurface of  a projective space of degree $m-1$ with a  point of multiplicity $m$ is  identically zero.

Therefore   (\ref{formula2indhomcase}) reduces to 
\begin{equation} \label{equalgencasehom1} h^0(X_s^n,D) =  h^0(X_s^n, D + E_1). \end{equation}

 One can verify that $D+E_1$ defines a redundant linear system, since its multiplicity at $p_1$ is $m-1$ which is strictly smaller than $k_C(D+E_1)=m$. Hence, by Remark \ref{redundant}  we have
\begin{equation} \label{conghomcase} H^0(X_s^n, D+E_1) \cong H^0(X_{s-1}^n, D+m E_1). \end{equation}
 Observe that 
$D_{s-1} := D+m E_1 = dH - \sum_{i=2}^s m E_i , $
 is a homogeneous divisor on $X^n_{s-1}$. Moreover, from
$$k_C = \frac{sm - nd + \varepsilon}{s-n-2} = m, $$
 with $\varepsilon < s-n-3$, we obtain
$$ k_C(D_{s-1}) = \frac{(s-1)m -nd + \varepsilon}{(s-1)-n-2} = m. $$
In particular we can apply  induction on $s$, and, using (\ref{equalgencasehom1}) and the isomorphism (\ref{conghomcase}), we compute that 
\begin{align} 
 h^0(X_s^n,D) 
		       & = \sum_{1\notin I} (-1)^{|I|}F_t(n+k_{I,\sigma_t} - r_{I,\sigma_t} - 1, s-1, \varepsilon, n). \label{homogeneous}
\end{align}
where,  we use $1\notin I$ to indicate that the sum ranges over all the indices $0\leq t \leq n/2$ and $I\subseteq \{2, \ldots, s\}$ with $0\leq |I| \leq n-2t$. 
As it turns out,  many of the integers $k_{I,\sigma_t}$ will be non positive, so they will not give any contribution to the formula.
 Indeed, first of all notice that for $I, I' \subseteq \{1, \ldots, s\}$ and $t, t' \in \mathbb N$ such that
$ |I| + t = |I'| + t',$
 then
\begin{equation} \label{samemultiplicities}
k_{I,\sigma_t} = (|I| + t)m - (|I| + t-1)d = k_{I', \sigma_{t'}}. 
\end{equation}
If $k_{\sigma_t}>0$, then $\sigma_t$ would be contained in the base locus of $|D|$ at least $k_{\sigma_t}$ times, but if $t > n/2$ then $\sigma_t$ fills up the ambient space and this would contradict the effectivity of $D$. Thus, for $t > n/2$, $k_{\sigma_t}$ must be negative or zero. 
 As a consequence, by (\ref{samemultiplicities}), for any index set $I\subseteq \{1, \ldots, s\}$ and $t\in \mathbb N$ such that $|I| + t > n/2$, the multiplicity of $J(L_I, \sigma_t)$ is
$$ k_{I,\sigma_t}  = k_{\sigma_{|I|+t}} \leq 0. $$
 This implies that in the summation (\ref{homogeneous}), all terms that do not satisfy $0\leq |I| + t \leq n/2 $  vanish. 

We shall separate the terms in  (\ref{homogeneous}) with $t=0$ from those with  $t\geq 1$ so that
\begin{equation} \label{homeq} h^0(X_s^n,D) = B_0 + B_1, \end{equation}
 where
\begin{align*}B_0 &= \sum_{\substack{0\leq |I| \leq n/2 \\ 1 \notin I}} (-1)^{|I|}F_0(n+k_I-r_I-1, s-1, \varepsilon, n), \\
B_1 & = \sum_{\substack{0\leq |I| + t \leq n/2 \\ t\geq 1, \; 1 \notin I}} (-1)^{|I|}F_t(n+k_{I,\sigma_t} - r_{I,\sigma_t} -1, s-1, \varepsilon, n). 
\end{align*}
Since the functions $F_0$ are independent of the parameter $s$, by its definition (\ref{F_0}), we may make the following substitutions in the summands of $B_0$:
$$ F_0(n+k_I-r_I-1, s-1, \varepsilon, n) = F_0(n+k_I-r_I-1, s, \varepsilon, n).$$
Applying Property \ref{property2} to the terms of $B_1$ we obtain
\begin{align*}
F_t(n+k_{I,\sigma_t} - r_{I,\sigma_t} -1, s-1, \varepsilon, n)  = & F_t(n+k_{I,\sigma_t} - r_{I,\sigma_t} -1, s, \varepsilon, n) \\
											& - F_t(n+k_{I,\sigma_t} - r_{I,\sigma_t} -1+1, s, \varepsilon, n).
\end{align*}
 Plugging everything into (\ref{homeq}), we get that
\begin{align}
h^0(X_s^n,D) = & \sum_{\substack{0\leq |I| \leq n/2 \\ 1 \notin I}} (-1)^{|I|}F_0(n+k_I-r_I-1, s, \varepsilon, n) \nonumber\\
		&+ \sum_{\substack{0\leq |I| + t \leq n/2 \\ t\geq 1, \; 1 \notin I}} (-1)^{|I|}F_t(n+k_{I,\sigma_t} - r_{I,\sigma_t} -1, s, \varepsilon, n) \label{homsum3} \\
		&- \sum_{\substack{0\leq |I| + t \leq n/2 \\ t\geq 1, \; 1 \notin I}} (-1)^{|I|}F_{t-1}(n+k_{I,\sigma_t} - r_{I,\sigma_t} -1+1, s, \varepsilon, n). \nonumber
\end{align}
We now look at the summation in the third line of (\ref{homsum3}):  since for any $t \geq 1$ and for any  $I\subseteq \{1, \ldots, s\}$ such that $1 \notin I $, by (\ref{samemultiplicities}) we have that
$k_{I,\sigma_t} = k_{I\cup \{1\}, \sigma_{t-1}}$ and that $r_{I,\sigma_t} = r_{I\cup\{1\}, \sigma_{t-1}}+1,$
then the summation can be written as
$$  \sum_{\substack{0\leq |I| + t \leq n/2 \\ t\geq 1, \; 1 \notin I}} (-1)^{|I|}F_{t-1}(n+k_{I\cup \{1 \},\sigma_{t-1}} - r_{I\cup \{1 \},\sigma_{t-1}} -1, s, \varepsilon, n)$$
$$ =  \sum_{\substack{0\leq |I| + t \leq n/2 \\ 1 \in I}} (-1)^{|I|-1}F_t(n+k_{I,\sigma_t} - r_{I,\sigma_t} -1, s, \varepsilon, n). $$
 Replacing this latter expression in (\ref{homsum3}) and adding everything together concludes the proof.
\end{proof}

We are now ready to complete the proof of the main theorem, Theorem \ref{ThmC}

\begin{proof}[Proof of Theorem \ref{ThmC}]  
 We shall assume that  $n\geq 3$, $k_C\geq 1$, $\varepsilon < s-n-3$, and that 
there is at least one index $i\in \{1, \ldots, s\}$ that satisfies $m_i > k_C + 1$.
In fact other cases are covered in Propositions \ref{ThmdimP2}-\ref{propB}-\ref{propC}-\ref{propD}. 
   Without loss of generality, we may assume that $m_1 \geq \ldots \geq m_s \geq k_C$ and $m_1 \geq k_C +1$. 
Since most of the proof follows ideas carried out in the proofs of Propositions \ref{propB}, \ref{propC} and \ref{propD},  we will not give full details but rather just a sketch of the argument, spotlighting only  the calculations that differ from the ones seen before. 

The inductive step is (\ref{indstep}):
\begin{equation} \label{indstepgencase} h^0(X_s^n, D)= h^0(X_s^n, D+E_1) - h^0(X_s^{n-1},D_{n-1}), \end{equation}
 where 
$$ D_{n-1} =  l(D+E_1) - k_C^+(D+E_1)E_q. $$
 First of all, it is straightforward  that, since $\varepsilon < s-n-3$ and $m_1 \geq k_C +1 \geq 2$, then
$$k_C(D+E_1) = \frac{ \sum_{i=1}^s m_i - nd -1 + (\varepsilon + 1)}{s-n-2} =k_C, $$
so that $\varepsilon(D+E_1)=\varepsilon(D) + 1$. Even more, from the non-redundance of $D$ and the fact that $m_1\geq k_C+1$, it is easy to see that $D+E_1$ is non-redundant either. Therefore, by backward induction on $\varepsilon$, we have
$$h^0(X_s^n,D+E_1) =  \sum_{\substack{0\leq t \leq n/2 \\ 0\leq |I| \leq n-2t}} (-1)^{|I|}F_t (n + k_{I, \sigma_t}(D+E_1) - r_{I, \sigma_t} - 1, s, \varepsilon +1 , n). $$
 We rewrite this by means of the following relations between integers \[
k_{I, \sigma_t}(D+E_1) =
	\begin{cases}
		k_{I, \sigma_t} &\quad \text{if } 1 \notin I \\
		k_{I, \sigma_t}-1 & \quad \text{if } 1 \in I,
	\end{cases}
\]
 which yields
\begin{align}
h^0(X_s^n,D+E_1) =	 & \sum_{\substack{0\leq t \leq n/2 \\ 0\leq |I| \leq n-2t \\ 1 \notin I}} (-1)^{|I|}F_t (n + k_{I, \sigma_t} - r_{I, \sigma_t} - 1, s, \varepsilon+1, n) \nonumber \\
				& + \sum_{\substack{0\leq t \leq (n-1)/2 \\ 0\leq |I| \leq n-2t \\ 1 \in I}} (-1)^{|I|}F_t (n + k_{I, \sigma_t} -1 - r_{I, \sigma_t} - 1, s, \varepsilon+1, n). \label{sum1firsgencase}
\end{align}

 We now look at the second summand of (\ref{indstepgencase}). We have that
$$D_{n-1} = (m_1 - 1)H_{n-1} - k_C E_{1,n-1} - \sum_{i=2}^s (m_i + m_1 - 1 -d) E_{i, n-1}. $$
By running a similar calculation to that of (\ref{kcn-1eps}), but recalling that in this case $\varepsilon < s-n-3$, we get
\begin{equation} \label{kcdn-1gencase} k_C(D_{n-1}) = k_{\{1\},\sigma_1} -1, \end{equation}
and $\varepsilon(D_{n-1}) = \varepsilon +1$. Furthermore, as explained after (\ref{kcn-1eps}), it follows from the effectivity and non-redundance of $D$ that $D_{n-1}$ is non-redundant either. 
 Hence, by induction on $n$, the dimension of $H^0(X_s^{n-1},D_{n-1})$ is \begin{equation} \label{Dn-1gencase} h^0(X_s^{n-1},D_{n-1}) =  \sum_{\substack{0\leq t \leq (n-1)/2 \\ 0\leq |I| \leq n-1-2t}} (-1)^{|I|}F_t (n-1 + k_{I, \sigma_t}(D_{n-1}) - r_{I, \sigma_t} - 1, s, \varepsilon +1 , n-1). \end{equation}
Using (\ref{kcdn-1gencase}), we obtain
\begin{equation*} \label{rewritegencase1}
k_{I, \sigma_t}(D_{n-1}) =
	\begin{cases}
		k_{I\cup \{1\}, \sigma_t}-1 &\quad \text{if } 1 \notin I \\
		k_{I\setminus \{1\}, \sigma_{t+1}} & \quad \text{if } 1 \in I.
	\end{cases}
\end{equation*}
It is also easy to see that 
\[
r_{I, \sigma_t} =
	\begin{cases}
		r_{I\cup \{s\}, \sigma_t}-1 &\quad \text{if } s \notin I \\
		r_{I\setminus \{s\}, \sigma_{t+1}}-1 & \quad \text{if } s \in I.
	\end{cases}
\]
Substituting in (\ref{Dn-1gencase}), we obtain
\begin{align*}
h^0(X_s^{n-1}, D_{n-1}) =&
 \sum_{\substack{0\leq t \leq (n-1)/2 \\ 0\leq |I| \leq n-2t \\ 1 \in I}}  (-1)^{|I|-1}F_t (n + k_{I, \sigma_t} -1 - r_{I, \sigma_t} - 1, s, \varepsilon+1, n-1) \\
& + \sum_{\substack{1 \leq t \leq n/2 \\ 0 \leq |I| \leq n-2t \\ 1 \notin I }}  (-1)^{|I|-1}F_{t-1} (n + k_{I, \sigma_t} - r_{I, \sigma_t} - 1, s, \varepsilon+1, n-1). \end{align*}
We will now show that by plugging the latter expression and (\ref{sum1firsgencase})  into (\ref{indstepgencase}) we conclude the proof. To see this, we separate the terms with $t=0$ of the first sum of (\ref{sum1firsgencase}) from those with  $1\leq t \leq n/2 $ obtaining \begin{equation} \label{generalcasefnal} h^0(X,D) =  S_0 + S_1 + S_2,  \end{equation}
 where
$$S_0 = \sum_{\substack{0 \leq |I| \leq n-2t \\ 1 \notin I }}    (-1)^{|I|}F_0 (n + k_{I, \sigma_t} - r_{I, \sigma_t} - 1, s, \varepsilon+1, n), \quad \quad \quad$$
\begin{align*}
S_1 =  &\sum_{\substack{1 \leq t \leq n/2 \\ 0 \leq |I| \leq n-2t \\ 1 \notin I }}    (-1)^{|I|}F_t (n + k_{I, \sigma_t} - r_{I, \sigma_t} - 1, s, \varepsilon+1, n) \\
		&+ \sum_{\substack{1 \leq t \leq n/2 \\ 0 \leq |I| \leq n-2t \\ 1 \notin I }}    (-1)^{|I|}F_{t-1} (n + k_{I, \sigma_t} - r_{I, \sigma_t} - 1, s, \varepsilon+1, n-1),
\end{align*}
\begin{align*}
S_2   =   &  \sum_{\substack{0\leq t \leq (n-1)/2 \\ 0\leq |I| \leq n-2t \\ 1 \in I}} (-1)^{|I|}F_t (n + k_{I, \sigma_t}  -1 - r_{I, \sigma_t} - 1, s, \varepsilon+1, n) \\
 		& +  \sum_{\substack{0\leq t \leq (n-1)/2 \\ 0\leq |I| \leq n-2t \\ 1 \in I}} (-1)^{|I|}F_t (n + k_{I, \sigma_t}  -1 - r_{I, \sigma_t} - 1, s, \varepsilon+1, n-1) .
\end{align*}
Now, $F_0$ is independent of $\varepsilon$. Therefore, in $S_0$ we can make the following replacement
$$ F_0 (n + k_{I, \sigma_t} - r_{I, \sigma_t} - 1, s, \varepsilon+1, n)= F_0 (n + k_{I, \sigma_t} - r_{I, \sigma_t} - 1, s, \varepsilon, n). $$
Applying Property \ref{property6} to $S_1$ and Property \ref{property7} to $S_2$, we obtain
\begin{align*} 
S_1& =\sum_{\substack{ 1\leq t \leq n/2 \\ 0\leq |I| \leq n-2t \\ 1 \notin I }}    (-1)^{|I|}F_t (n + k_{I, \sigma_t} - r_{I, \sigma_t} - 1, s, \varepsilon, n), \\
S_2& =\sum_{\substack{0 \leq t \leq (n-1)/2 \\ 0\leq |I| \leq n-2t \\ 1 \in I}}   (-1)^{|I|}F_t (n + k_{I, \sigma_t}  - r_{I, \sigma_t} - 1, \varepsilon, n).  \end{align*}
 Putting everything together in (\ref{generalcasefnal}), we conclude that
$$ h^0(X,D) = 	\sum_{\substack{0 \leq t \leq n/2 \\ 0\leq |I| \leq n-2t}} (-1)^{|I|}F_t (n + k_{I, \sigma_t} - r_{I, \sigma_t} - 1, s,\varepsilon, n). $$

\end{proof}


\appendix
\section{}\label{Fiveproperties}

In this section we prove  the five properties of the function $F_t$, stated in Section \ref{ProofMainTheorem}. It is convenient to recall the definition of 
the function here:
\begin{equation} \label{F_tappendix} F_t(a, s, \varepsilon , n) := \binom{a}{n} + \sum_{i=1}^t \binom{s-n-4 + i}{i} \binom{a+i}{n} - \sum_{i=1}^t \binom{\varepsilon}{i}F_{t-i}(a,s,\varepsilon, n-i). \end{equation}
Before we proceed with the proof of the properties hold, we  need  to 
establish a further identity.

\begin{lemma} \label{property1} For every $t, s, \varepsilon, n \in \mathbb N$ and for every $a \in \mathbb Z$, we have
$$ F_t(a,s,\varepsilon, n)=F_t(a-1, s, \varepsilon, n) + F_t(a-1, s-1, \varepsilon, n-1).$$
\end{lemma}
\begin{proof}
The proof follows by a simple induction on $t$ and, as in many of these properties, by exploiting the standard identity of binomials
\begin{equation} \label{eq1} \binom{a}{n} = \binom{a-1}{n} + \binom{a-1}{n-1}.\end{equation}
In this case, \eqref{eq1} establishes the case $t=0$.  We we fix $t$ and we assume the statement holds for all integers up to $t-1$. In particular, for every $i \in \{1, \ldots, t\}$,  we have
\begin{equation} \label{eq3} F_{t-i}(a,s, \varepsilon, n-i) = F_{t-i}(a-1,s,\varepsilon, n-i) +  F_{t-i}(a-1,s-1,\varepsilon, n-1-i). \end{equation}
Moreover the following equality holds:
\begin{equation} \label{eq2} \binom{a+i}{n} = \binom{a+i-1}{n} + \binom{a+i-1}{n-1}. \end{equation}
We conclude by substituting (\ref{eq1}), (\ref{eq2}) and (\ref{eq3}) in the definition of $F_t(a,s,\varepsilon, n)$.
\end{proof}

\ \\
\textbf{Property \ref{property2}}. 
$F_t(a,s-1,\varepsilon, n)  + F_{t-1}(a+1, s, \varepsilon, n)=F_t(a,s, \varepsilon, n)$, for every $t, s, \varepsilon, n \in \mathbb N$ with $t\geq 1$, and every $a \in \mathbb Z$.
\begin{proof}
We use induction on $t$. The case $t=1$ follows immediately. Assume that the property is true up to $t-1$, then
\begin{align*}
F_t(a,s,\varepsilon, n) 	= 	& \binom{a}{n} + \sum_{i=1}^t \binom{s-n-4 + i}{i} \binom{a+i}{n} - \sum_{i=1}^t \binom{\varepsilon}{i}F_{t-i}(a,s,\varepsilon, n-i)  
\end{align*}
\begin{align*}
\;\;	\quad \quad\quad \quad = 	& \binom{a}{n} + \sum_{i=1}^t \left(\binom{s-1-n-4 + i}{i} + \binom{s-1-n-4 + i}{i-1} \right) \binom{a+i}{n} \\
						&- \sum_{i=1}^{t-1} \binom{\varepsilon}{i}\left(F_{t-i}(a,s-1,\varepsilon, n-i)  + F_{t-1-i}(a+1, s, \varepsilon, n-i) \right) \\
					        &- \binom{\varepsilon}{t} F_0(a,s,\varepsilon, n-t) .		
\end{align*}
 Rearranging this expression and observing that $F_0$ is independent of $s$ finishes the proof. 

\end{proof}

The following property allows to rewrite $F_t$ when $\varepsilon = s-n-3$.
\begin{lemma} \label{property3} For every $t, a, s, \varepsilon, n \in \mathbb N$, with $s\geq n+3$, we have
$$F_t(a,s,s-n-3,n) = \binom{a}{n} + \sum_{i=1}^t\binom{s-n-4+i}{i} \binom{a}{n}. $$
\end{lemma}
\begin{proof}
 We prove it by double induction on $n$ and $a$.
 
 The case $n=0$ is obviously true. 
If $n=1$ we have that
\begin{align*}
F_t(a,s,s-4,1) = & \ a + \sum_{i=1}^t \binom{s-5+i}{i}(a+i) - (s-4)\left(1 + \sum_{i=1}^{t-1}\binom{s-4+i}{i} \right)\\
		     = & \ a + \sum_{i=1}^t \binom{s-5+i}{i}a+ \sum_{i=1}^t \binom{s-5+i}{i}i - (s-4) \binom{s-4+t}{t-1} .
\end{align*}
 We claim that
$$ \sum_{i=1}^t \binom{s-5+i}{i}i - (s-4) \binom{s-4+t}{t-1}  = 0 , $$
 which implies that the property holds for $n=1$. We  prove the claim by induction on $t$. It clearly holds for $t=1$. Assuming that the claim holds for $t-1$, we have
\begin{align*}
 \sum_{i=1}^t \binom{s-5+i}{i}i - (s-4) \binom{s-4+t}{t-1}   
 =& \left(\sum_{i=1}^{t-1} \binom{s-5+i}{i} i - (s-4)\binom{s-4+t-1}{t-2}\right) \\
& + \left( \binom{s-5+t}{t}t - (s-4)\binom{s-4+t-1}{t-1} \right) \\
=& \ 0 + \left(\frac{(s-5+t)!}{(t-1)!(s-5)!} - \frac{(s-5+t)!}{(t-1)!(s-5)!}\right) \\
=&\ 0. \end{align*}
This concludes the proof of the claim. From now on we may assume the property holds for up to $n-1$. 

To apply induction on $a$, we first notice that the statement is true for any $a<n$ since for those cases 
$F_t(a,s, s-n-3, n) = 0. $ Indeed, Lemma \ref{property1} implies that $$ F_t(a, s, s-n-3,n) = F_t(a-1, s, s-n-3, n) + F_t(a-1, s-1, s-n-3, n-1). $$
 The second summand vanishes by the induction hypothesis on $n$. Thus, we can apply this recursively until we have that $a<n-t$. It is easy to see from the definition of $F_t$ that for any such $a<n-t$, we have
$F_t(a,s,\varepsilon, n) = 0, $
 for any value of $\varepsilon$, so in particular also for $\varepsilon=s-n-3$.

 Let us now prove the case $a=n$. Again, applying Lemma \ref{property1} we get
$$ F_t(n,s,s-n-3,n) = F_t(n-1, s , s-n-3, n) + F_t(n-1, s-1, s-n-3, n-1). $$
 The first summand vanishes from the previous discussion; as for the second summand, the induction hypothesis on $n$ gives
\begin{align*}
 F_t(n,s,s-n-3,n) = & \binom{n-1}{n-1} + \sum_{i=1}^t \binom{s-1-(n-1) - 4 +i}{i}\binom{n-1}{n-1} \\
	                = & \binom{n}{n} + \sum_{i=1}^t \binom{s-n - 4 +i}{i}\binom{n}{n}.
\end{align*}
This proves the claim. 

Let us now assume that the property holds for $a-1$. By Lemma \ref{property1} and the induction hypothesis on $a$ and $n$, we conclude that
\begin{align*}
 F_t(a,s,s-n-3,n) = &
  F_t(a-1,s,s-n-3,n) + F_t(a-1,s-1,s-n-3, n-1)  \\
 = &\binom{a-1}{n} + \sum_{i=1}^t\binom{s-n-4+i}{i} \binom{a-1}{n} \\
& + \binom{a-1}{n-1} + \sum_{i=1}^t \binom{s-n-4+i}{i} \binom{a-1}{n-1} \\
= & \binom{a}{n} + \sum_{i=1}^t \binom{s-n-4+i}{i} \binom{a}{n}. \end{align*}
\end{proof}

\vspace{5mm}

In  \eqref{hockey} we introduced a useful binomial identity called the hockey-stick identity. The mirror image of this identity on Pascal's triangle  consists of a reformulation of the latter that will be useful proving Property \ref{property4}. This is often referred to as  the \emph{Christmas-stocking identity}. 

\begin{lemma} \label{christmas} For any two natural numbers $b, t$, we have 
$$ \sum_{i=0}^{t-1} \binom{b+ i }{i} = \binom{b + t}{t-1}. $$
\end{lemma}
\ \\
\textbf{Property \ref{property4}.}  $F_t(a,s, 0, n) + F_{t-1}(a, s, 0, n-1) = F_t(a+t, s, s-n-3, n)$, for every $t, a, s, n \in \mathbb N$, with $t\geq 1$ and $s\geq n+3$.
\begin{proof}
We will prove that
$$F_t(a,s, 0, n) + F_{t-1}(a, s, 0, n-1) = \binom{a+t}{n} + \sum_{i=1}^t\binom{s-n-4+i}{i} \binom{a+t}{n}, $$
 and then conclude by Lemma \ref{property3}. We proceed by induction on $t$. It clearly holds for $t=1$. Thus, we assume that the property holds for $t-1$.

Since in this case $\varepsilon=0$,    the following identity follows from \eqref{F_tappendix}:
$$ F_t(a,s, 0, n)=F_{t-1}(a,s,0,n) + \binom{s-n-4+t}{t} \binom{a+t}{n}.$$
Applying the same identity  to the term $ F_{t-1}(a,s, 0, n-1)$, we obtain
%
$$ F_t(a,s, 0, n) + F_{t-1}(a, s, 0, n-1) = $$
$$ = F_{t-1}(a,s,0,n) + \binom{s-n-4+t}{t} \binom{a+t}{n} $$
$$  + F_{t-2}(a, s, 0, n-1) + \binom{s-n-4+t}{t-1} \binom{a+t-1}{n-1}. $$
 Applying the induction hypothesis on $t$ in the latter expression, and  using Lemma \ref{christmas}, we obtain
$$ F_t(a,s, 0, n) + F_{t-1}(a, s, 0, n-1) = $$
$$ =  F_{t-1}(a+t-1, s, s-n-3, n) + \binom{s-n-4+t}{t} \binom{a+t}{n}. $$
$$ + \left(1 + \sum_{i=1}^{t-1} \binom{s-n-4+i}{i} \right) \binom{a+t-1}{n-1}. $$
 Finally, by applying Lemma \ref{property3} to the latter expression, we arrive to
$$ F_t(a,s, 0, n) + F_{t-1}(a, s, 0, n-1) = $$
$$ = \binom{a+t-1}{n} + \sum_{i=1}^{t-1} \binom{s-n-4+i}{i} \binom{a+t-1}{n} $$
$$ + \binom{s-n-4+t}{t} \binom{a+t}{n} $$
$$  + \left(1 + \sum_{i=1}^{t-1} \binom{s-n-4+i}{i} \right) \binom{a+t-1}{n-1} $$
$$   =  \binom{a+t}{n} + \sum_{i=1}^{t} \binom{s-n-4+i}{i} \binom{a+t}{n}  =   F_t(a+t, s, s-n-3, n). $$
\end{proof}
\ \\
\textbf{Property  \ref{property5}.} 
$F_t(a,s, 0, n) + F_{t}(a, s, 0, n-1) = F_t(a+t+1, s, s-n-3, n),$ for every $t, a, s, n \in \mathbb N$ with $s\geq n+3$.
\begin{proof}
This property follows from Lemma \ref{property3}, Property \ref{property4} and Lemma \ref{christmas}. Indeed,
$$ F_t(a,s,0,n) + F_t(a,s,0,n-1) = $$
$$ = F_t(a,s,0,n) + F_{t-1}(a,s,0,n-1) + \binom{s-n-4+t+1}{t} \binom{a+t}{n-1} $$ 
$$ =\binom{a+t}{n} + \sum_{i=1}^{t} \binom{s-n-4+i}{i} \binom{a+t}{n} + \left(1+ \sum_{i=1}^t \binom{s-n-4+i}{i} \right) \binom{a+t}{n-1} $$
$$ = \binom{a+t+1}{n} + \sum_{i=1}^{t} \binom{s-n-4+i}{i} \binom{a+t+1}{n}  = F_t(a+t+1, s, s-n-3, n). $$
\end{proof}
\ \\ 
\textbf{Property \ref{property6}.} 
$F_t(a,s, \varepsilon, n) + F_{t-1}(a,s, \varepsilon, n-1) = F_t(a, s, \varepsilon -1,n)$, for all $t, a, s, \varepsilon, n \in \mathbb{N}$ with  $\varepsilon \neq 0$.
\begin{proof}
We will prove it by induction on $t$. The case $t=1$ follows easily. Let us assume the property is true for up to $t-1$. By expanding the expression of each function according to its definition \eqref{F_tappendix}, and rearranging some terms, we arrive to
$$  F_t(a,s,\varepsilon, n) + F_{t-1}(a, s, \varepsilon, n-1) 	=  $$
$$ =  \binom{a}{n} + \sum_{i=1}^t \binom{s-n-4 + i}{i} \binom{a+i}{n} $$
$$  - \sum_{i=1}^{t-1} \binom{\varepsilon}{i} \left(F_{t-i}(a,s,\varepsilon, n-i) +F_{(t-1)-i}(a,s,\varepsilon, (n-1)-i) \right)  $$
$$  + \binom{\varepsilon}{t}F_0(a,s,\varepsilon, n-t) +  \binom{a}{n-1} + \sum_{i=1}^{t-1}\binom{s-(n-1)-4 + i}{i} \binom{a+i}{n-1} .$$
 Applying the induction hypothesis on $t$ and recalling that $F_0$ is independent from $\varepsilon$, we get
$$  F_t(a,s,\varepsilon, n) + F_{t-1}(a, s, \varepsilon, n-1) 	=  $$
$$ =  \binom{a}{n} + \sum_{i=1}^t \binom{s-n-4 + i}{i} \binom{a+i}{n} - \sum_{i=1}^{t} \binom{\varepsilon}{i} F_{t-i}(a,s,\varepsilon-1, n-i)  $$
$$ + \binom{a}{n-1} + \sum_{i=1}^{t-1}\binom{s-(n-1)-4 + i}{i} \binom{a+i}{n-1},  $$
To conclude the proof we use the fact that
$$ \binom{\varepsilon}{i} = \binom{\varepsilon -1}{i} + \binom{\varepsilon -1}{i-1}. $$
Substituting the latter in the equation above,  we have
$$  F_t(a,s,\varepsilon, n) + F_{t-1}(a, s, \varepsilon, n-1) 	=  $$
$$ =  \left( \binom{a}{n} + \sum_{i=1}^t \binom{s-n-4 + i}{i} \binom{a+i}{n} - \sum_{i=1}^{t} \binom{\varepsilon-1}{i} F_{t-i}(a,s,\varepsilon-1, n-i)  \right)$$
$$ - \sum_{i=1}^{t} \binom{\varepsilon-1}{i-1} F_{t-i}(a,s,\varepsilon-1, n-i) + \binom{a}{n-1} + \sum_{i=1}^{t-1}\binom{s-(n-1)-4 + i}{i} \binom{a+i}{n-1} $$
\vspace{2mm}
$$ = F_t(a,s,\varepsilon-1,n) - F_{t-1}(a,s,\varepsilon-1, n-1) - \sum_{i=2}^{t}\binom{\varepsilon-1}{i-1}F_{t-i}(a,s,\varepsilon-1,n-i) $$
$$ + \binom{a}{n-1} + \sum_{i=1}^{t-1}\binom{s-(n-1)-4 + i}{i} \binom{a+i}{n-1} = $$
\vspace{2mm}
$$ = F_t(a,s,\varepsilon-1,n) - F_{t-1}(a,s,\varepsilon-1, n-1) - \sum_{i=1}^{t-1}\binom{\varepsilon-1}{i}F_{t-1-i}(a,s,\varepsilon-1,n-1-i) $$
$$ + \binom{a}{n-1} + \sum_{i=1}^{t-1}\binom{s-(n-1)-4 + i}{i} \binom{a+i}{n-1} = $$
\vspace{2mm}
$$= F_t(a,s,\varepsilon-1,n) - F_{t-1}(a,s,\varepsilon-1, n-1) +F_{t-1}(a,s,\varepsilon-1, n-1) =  F_t(a,s,\varepsilon-1,n). $$
\end{proof}
\ \\ 
\textbf{Property \ref{property7}.}
$F_t(a,s, \varepsilon, n) + F_t(a, s, \varepsilon, n-1) = F_t(a+1, s, \varepsilon -1, n)$, for every $t, a, s, \varepsilon, n \in \mathbb{N}$ with $\varepsilon \neq 0 $.
\begin{proof}
The proof is very similar to the proof of Property \ref{property6}and  it is by induct ion on $t$. For $t=0$ the result is obviously true, and it easily follows for $t=1$. We now assume the property to be true up to $t-1$. By expanding each function according to its definition \eqref{F_tappendix} and using the induction hypothesis we see that
$$F_t(a,s,\varepsilon,n) + F_t(a,s,\varepsilon,n-1) = $$
$$ =  \binom{a}{n} + \sum_{i=1}^t \binom{s-n-4 + i}{i} \binom{a+i}{n} - \sum_{i=1}^t \binom{\varepsilon}{i}F_{t-i}(a+1,s,\varepsilon-1, n-i) $$
$$ +  \binom{a}{n-1} + \sum_{i=1}^t \binom{s-n-4 + i +1}{i} \binom{a+i}{n-1}. $$
 We now rewrite
$$ \binom{s-n-4+i+1}{i} = \binom{s-n-4+i}{i} + \binom{s-n-4+i}{i-1}, $$
 and for all $i \in \{0, \ldots, t\}$
$$ \binom{a+i}{n} + \binom{a+i}{n-1}= \binom{a+1+i}{n}. $$ 
 Thus, we arrive to
$$F_t(a,s,\varepsilon,n) + F_t(a,s,\varepsilon,n-1) = $$
$$ =  \binom{a+1}{n} + \sum_{i=1}^t \binom{s-n-4 + i}{i} \binom{a+1+i}{n} - \sum_{i=1}^t \binom{\varepsilon}{i}F_{t-i}(a+1,s,\varepsilon-1, n-i) $$
$$ + \binom{a+1}{n-1} + \sum_{i=1}^{t-1} \binom{s-n-4 + i +1}{i} \binom{a+1+i}{n-1}. $$
 Finally, as in the proof of Property \ref{property6}, we conclude by using the identity
$$ \binom{\varepsilon}{i} = \binom{\varepsilon -1}{i} + \binom{\varepsilon -1}{i-1} $$
and rearranging terms. Indeed, by doing so one arrives to
$$F_t(a,s,\varepsilon,n) + F_t(a,s,\varepsilon,n-1) = $$
$$ =  F_t(a+1,s,\varepsilon-1,n) - F_{t-1}(a+1, s, \varepsilon -1, n-1) + F_{t-1}(a+1, s, \varepsilon -1, n-1)$$
$$ =   F_t(a+1,s,\varepsilon-1,n). $$
\end{proof}


\end{document}